\pdfoutput=1
\documentclass[a4paper,reqno,a4paper]{amsart}
\usepackage[T1]{fontenc}
\usepackage[utf8]{inputenc}
\usepackage{color}
\usepackage{prettyref}
\usepackage{units}
\usepackage{mathrsfs}
\usepackage{mathtools}
\usepackage{amstext}
\usepackage{amsthm}
\usepackage{amssymb}
\usepackage{stmaryrd}
\usepackage{microtype}
\usepackage[bookmarks=false,
 breaklinks=true,pdfborder={0 0 0},pdfborderstyle={},backref=false,colorlinks=true]
 {hyperref}
\hypersetup{
 pdfstartview=FitV,linkcolor=black,citecolor=black}

\makeatletter

\pdfpageheight\paperheight
\pdfpagewidth\paperwidth

\numberwithin{equation}{section}
\numberwithin{figure}{section}
\theoremstyle{plain}
\newtheorem{thm}{\protect\theoremname}[section]
\theoremstyle{remark}
\newtheorem{rem}[thm]{\protect\remarkname}
\theoremstyle{plain}
\newtheorem{lem}[thm]{\protect\lemmaname}
\theoremstyle{plain}
\newtheorem{cor}[thm]{\protect\corollaryname}
\theoremstyle{plain}
\newtheorem{prop}[thm]{\protect\propositionname}
\theoremstyle{definition}
\newtheorem{defn}[thm]{\protect\definitionname}
\theoremstyle{definition}
\newtheorem{example}[thm]{\protect\examplename}

\usepackage{bbm}
\usepackage{pgfplots}
\usepgfplotslibrary{colormaps,fillbetween}
\pgfplotsset{compat=1.15}
\usepackage{mathrsfs}
\usetikzlibrary{arrows}
 
\usepackage{oldgerm}

\usepackage{txfonts}
\usepackage{textcomp}
\usepackage{xcolor}
\usepackage{dsfont}
\usepackage{graphicx}
\usepackage{caption}
\usepackage{subcaption}
\usepackage{float}
\usepackage{tikz}
\usepackage[foot]{amsaddr}
\usepackage{enumitem}

\renewcommand{\d}{\mathrm{d}}
\renewcommand{\rho}{\varrho}

\newcommand{\1}{\mathbbm{1}}

\newcommand{\e}{\mathrm{e}} 
\newcommand{\D}{\mathcal{D}}
\newcommand{\DD}{{\mathcal{S}}}

\DeclareMathOperator{\supp}{supp}
\DeclareMathOperator{\card}{card}

\renewcommand{\tilde}{\widetilde}
\renewcommand{\emptyset}{\varnothing}
\renewcommand{\phi}{\varphi}
\renewcommand{\epsilon}{\varepsilon}
\def\N{\mathbb N}
\def\d{\;\mathrm d}
\def\R{\mathbb{R}}
\def\Z{\mathbb{Z}}

\def\Q{\mathcal{Q}}

\def\J{\mathfrak{J}}
\def\GL{\tau}
 
\definecolor{lime}{HTML}{A6CE39}
\DeclareRobustCommand{\orcidicon}{%
	\begin{tikzpicture}
	\draw[lime, fill=lime] (0,0) 
	circle [radius=0.16] 
	node[white] {{\fontfamily{qag}\selectfont \tiny ID}};
	\draw[white, fill=white] (-0.0625,0.095) 
	circle [radius=0.007];
	\end{tikzpicture}
	\hspace{-2mm}
}

\foreach \x in {A, ..., Z}{%
	\expandafter\xdef\csname orcid\x\endcsname{\noexpand\href{https://orcid.org/\csname orcidauthor\x\endcsname}{\noexpand\orcidicon}}
}

\AtBeginDocument{
\DeclareFieldFormat[article]{title}{{#1}}
\DeclareFieldFormat[phdthesis]{citetitle}{{\em #1}}
\DeclareFieldFormat[unpublished]{title}{{#1}}
\DeclareFieldFormat[proceedings]{title}{{#1}}
\DeclareFieldFormat[incollection]{title}{{#1}}
\DeclareFieldFormat[unpublished]{title}{{\em #1}}
\renewbibmacro{in:}{}
\DeclareFieldFormat[article]{pages}{#1}
}

\newrefformat{prop}{Proposition~\ref{#1}}
\newrefformat{rem}{Remark~\ref{#1}}
\newrefformat{fact}{Fact~\ref{#1}}
\newrefformat{exa}{Example~\ref{#1}}
\newrefformat{cor}{Corollary~\ref{#1}}
\newrefformat{assu}{Assumption~\ref{#1}}
\newrefformat{subsec}{Section~\ref{#1}}
\newrefformat{sec}{Section~\ref{#1}}
\newrefformat{def}{Definition~\ref{#1}}
\newrefformat{Algo}{Algorithm~\ref{#1}}

\newcommand{\q}{\mathfrak{q}}

\makeatother

\usepackage[style=authoryear,backend=bibtex,style=alphabetic,  backref=false, giveninits=true,citestyle=alphabetic, natbib=false, url=false, isbn=false,doi=true,eprint=false,maxbibnames=99]{biblatex}
\providecommand{\corollaryname}{Corollary}
\providecommand{\definitionname}{Definition}
\providecommand{\examplename}{Example}
\providecommand{\lemmaname}{Lemma}
\providecommand{\propositionname}{Proposition}
\providecommand{\remarkname}{Remark}
\providecommand{\theoremname}{Theorem}

\begin{document}
\title[Exact asymptotic order for adaptive approximations]{Exact asymptotic order for generalised adaptive approximations}
\author{Marc Kesseböhmer \orcidA{}}
\email{mhk@uni-bremen.de}
\author{Aljoscha Niemann}
\email{niemann1@uni-bremen.de}
\begin{abstract}
In this note, we present an abstract approach to study asymptotic
orders for adaptive approximations with respect to a monotone set
function $\J$ defined on dyadic cubes. We determine the exact upper
order in terms of the critical value of the corresponding $\J$-partition
function, and we are able to provide upper and lower bounds in term
of fractal-geometric quantities. With properly chosen $\J$, our new
approach has applications in many different areas of mathematics,
including the spectral theory of Kre\u{\i}n–Feller operators, quantization
dimensions of compactly supported probability measures, and the exact
asymptotic order for Kolmogorov, Gel\textquotesingle fand and linear
widths for Sobolev embeddings into $L_{\mu}^{p}$-spaces.
\end{abstract}

\address{Institute for Dynamical Systems, Faculty 3 – Mathematics and Computer
Science, University of Bremen, Bibliothekstr. 5, 28359 Bremen, Germany}
\keywords{adaptive approximation algorithm, approximation theory, $L^{q}$-spectrum,
partition functions, Minkowski dimension, (coarse) multifractal formalism.}
\subjclass[2000]{primary: 68W25, 41A25 secondary: 28A80, 65D15}
\thanks{This research was supported by the Deutsche Forschungsgemeinschaft
(DFG, German Research Foundation) grant Ke 1440/3-1 and as part of
the Collaborative Research Centre TRR 181 (L03) \emph{'Energy Transfers
in Atmosphere and Ocean}' funded by the DFG - Grant number 274762653.}
\maketitle

\section{Introduction and statement of main results}

The study of adaptive approximation algorithms goes back to the seminal
work of Birman and Solomjak in the 1970s \cite{MR0217487,MR0482138},
which was motivated by the study of asymptotics for spectral problems,
and was subsequently refined by Borzov 1971 \cite{Borzov1971} for
singular measures, and then by DeVore and Yu \cite{MR939183,MR1035930}
for certain boundary cases not treated by Birman, Solomjak or Borzov.
One of the earliest comprehensive treatments of such adaptive approximations
in a geometric context for the study of convex bodies can be found
in the in the textbook \cite{MR353117}. Generally speaking, we deal
with the asymptotics of counting problems derived from set functions
defined on dyadic subcubes of the unit cube. Recently, this problem
has attracted renewed attention in the context of
\begin{itemize}
\item \emph{piecewise polynomial approximation} in \cite{MR1781213,MR4077830},
\item \emph{spectral asymptotics} in \cite{MR4331823,MR4484835,KN2022b,KN2022,KN21},
\item \emph{quantization of probability measures} in \cite{KN22b,KN2024QDnegativeorder},
and
\item \emph{Kolmogorov, Gel\textquotesingle fand, }and\emph{ linear widths}
in \cite{MR4444736,KesseboehmerWiegmann}.
\end{itemize}
Our new approach improves some of the classic results (see e.\,g\@.
\prettyref{subsec:PartitionImprovemetBirmanSolomyak}, where we compare
our results with work of Birman and Solomjak from the 1970s) and is
fundamental for all the results by the authors mentioned above. In
this note we also allow a generalisation with respect to the range
of set functions considered, as this proves to be very useful for
applications to spectral asymptotics (e.\,g\@. in \cite{KN2022b}).
However, many applications involve set functions that are defined
on all dyadic cubes without further restrictions; we will refer to
this case as the \emph{classical case}.

\subsection{The basic setting\label{subsec:Introduction-and-basic-setting}}

This paper is concerned with the study of the asymptotic behaviour
of an adaptive approximation algorithm in the following setting. For
$d\in\N$, we call $Q\coloneqq I_{1}\times\cdots\times I_{d}$ a dyadic
cube of side length $2^{-n}$ if $I_{i}$ are (half-open, open, or
closed) intervals with endpoints in the dyadic grid $\left\{ k2^{-n}:k\in\Z\right\} $.
Inductively, we define a sequence $\mathcal{D}_{n}$ of dyadic partitions
as follows: Let $\Q\subset\R^{d}$ denote a particular choice of a
dyadic unit cube and set $\mathcal{D}_{0}\coloneqq\left\{ \Q\right\} $.
For $\mathcal{D}_{n}$ given, we let $\mathcal{D}_{n+1}$ be a refinement
of $\mathcal{D}_{n}$, this means that each element of $\mathcal{D}_{n}$
can be decomposed into $2^{d}$ disjoint elements of $\mathcal{D}_{n+1}$.
Note that cubes in $\mathcal{D}_{n}$ are not necessarily congruent,
in that we allow certain faces of $Q\in\mathcal{D}_{n}$ not to be
a subset of $Q$. In this way, each $\mathcal{D}_{n}$ defines a dyadic
partition of $\Q=\bigcup\mathcal{D}_{n}$, $n\in\N_{0}$, and the
union of all such partitions $\mathcal{D}\coloneqq\bigcup_{n\in\N_{0}}\mathcal{D}_{n}$
defines a \emph{semiring} of sets, that is $\mathcal{D}\cup\left\{ \emptyset\right\} $
is stable under intersections and for any $A,B\in\mathcal{D}$ with
$A\subset B$ we have that $B\setminus A$ can be written as a finite
disjoint union of sets from $\mathcal{D}$. For some applications
of our formalism, a more general approach is required (see \prettyref{rem:Our-most-prominent}).
For this we will introduce a fixed subset $\DD\subset\D$ and set
its level $n\in\N_{0}$ cubes to be $\DD{}_{n}\coloneqq\DD\cap\D_{n}$.
Throughout the paper, we also fix a set function 
\[
\J:\mathcal{\DD}\to\R_{\geq0},
\]
which is assumed to be
\begin{enumerate}
\item \emph{non-trivial, }i.\,e\@. $\J$ is not identically zero,
\item \emph{monotone}, i.\,e\@. for all $Q,Q'\in\DD$ with $Q\subset Q'$
we have $\J\left(Q\right)\leq\J\left(Q'\right)$,
\item \emph{uniformly vanishing} in the sense that ${\displaystyle j_{n}\coloneqq\sup_{Q\in\bigcup_{k\geq n}\DD_{k}}\J\left(Q\right)\searrow0}$,
for $n\to\infty$,
\item and \emph{locally non-vanishing}, i.\,e\@. for $n\in\N_{0}$ and
each cube $Q$ from $\DD_{n}$ with $\J\left(Q\right)>0$ there exists
at least one subcube $Q'\subset Q$ with $Q'\in\DD_{n+1}$ and $\J\left(Q'\right)>0$.
\end{enumerate}
Throughout, by $\Lambda$ we will denote the $d$-dimensional Lebesgue
measure. For $x>1/j_{0}$, we define the so-called \emph{minimal $x$-good
partition} by 
\begin{equation}
G_{x}\coloneqq\left\{ Q\in\DD:\J\left(Q\right)<1/x\,\&\,\exists Q'\in\DD_{\left|\log_{2}\Lambda\left(Q\right)\right|/d-1}:Q'\supset Q\;\&\;\J(Q')\geq1/x\right\} .\label{eq:Def_G_x}
\end{equation}
Note that, strictly speaking, $G_{x}$ is not a partition unless we
are in the classical case, i.\,e\@. $\DD=\D$. However, $G_{x}$
does partition the '$x$-bad cubes', i.\,e\@. the union of those
$Q\in\DD$ with $\J\left(Q\right)\geq1/x$, where we ignore cubes
from $\D\setminus\DD$. Also note that for all $Q\in\mathcal{D}$
and $n\in\N$ we have that $\left|\log_{2}\Lambda\left(Q\right)\right|/d=n$
if and only if $Q\in\mathcal{D}_{n}$.

The aim of this work is to investigate the growth rate of
\begin{equation}
M\left(x\right)\coloneqq\card\left(G_{x}\right)\label{eq:Def_M(x)}
\end{equation}
 as $x\in\R_{>0}$ tends to infinity, with regard to the leading exponents
\begin{equation}
\overline{h}\coloneqq\overline{h}_{\J}\coloneqq\limsup_{x\to\infty}\frac{\log\left(M\left(x\right)\right)}{\log\left(x\right)}\;\text{and }\:\underline{h}\coloneqq\underline{h}_{\J}\coloneqq\liminf_{x\to\infty}\frac{\log\left(M\left(x\right)\right)}{\log\left(x\right)}.\label{eq:Def_h_}
\end{equation}
We will refer to these quantities as the \emph{upper}, resp\@. \emph{lower},
$\J$\emph{-partition entropy}. In fact, we will determine the upper
$\J$-partition entropy in terms of the $\J$-partition function,
which generalises the concept of the $L^{q}$-spectrum for measures,
see \prettyref{subsec:-Partition-functions}. For the lower $\J$-partition
entropy we provide natural bounds, and for particularly regular cases
we can also determine its value. In any case, upper and lower bounds
are provided in terms of specific fractal quantities.
\begin{rem}
\label{rem:Our-most-prominent}Our most prominent example of such
a restriction to $\DD$ of the cubes of $\D$ can be found in \cite{KN2022b}.
In this example, we consider the set $\DD\coloneqq\left\{ Q\in\mathcal{D}:\partial\Q\cap\overline{Q}=\varnothing\right\} $
ignoring all cubes touching the boundary of $\Q$ in order to handle
the Dirichlet case for the spectral asymptotics of Kre\u{\i}n–Feller
operators in higher dimensions.
\end{rem}

\subsection{The dual problem\label{subsec:The-dual-problem}}

For applications (like for quantization of probability measures) the
dual problem is also sometimes useful. For convenience we write 
\[
\J\left(P\right)\coloneqq\max_{Q\in P}\J\left(Q\right)
\]
 for any collection of cubes $P\subset\DD$. With this at hand, we
define 
\begin{equation}
\gamma_{n}\coloneqq\gamma_{\J,n}\coloneqq\min_{P\in\Pi_{n}}\J\left(P\right),\:\text{where \ensuremath{\Pi_{n}\coloneqq\left\{  P=G_{x}:\text{for some \ensuremath{x>0} and }\card\left(P\right)\leq n\right\} } .}\label{eq:Def_gamma}
\end{equation}
We will investigate the following upper, resp. lower, exponent of
convergence of $\gamma_{n}$ given by
\begin{equation}
\overline{\alpha}\coloneqq\overline{\alpha}_{\J}\coloneqq\limsup_{n\rightarrow\infty}\frac{\log\left(\gamma_{n}\right)}{\log\left(n\right)},\qquad\text{resp.}\qquad\underline{\alpha}\coloneqq\underline{\alpha}_{\J}\coloneqq\liminf_{n\rightarrow\infty}\frac{\log\left(\gamma_{n}\right)}{\log\left(n\right)}.\label{eq:Def_alpha_}
\end{equation}

\subsection{The classical case ($\DD=\D)$ and the adaptive approximation algorithm}

This section is devoted to study the classical case $\DD=\D$ which
leads to the classical adaptive approximation algorithm studied intensively
in the past decades (see \cite{MR1035930,MR1781213,MR939183,Borzov1971,MR0209733}).
For this we show how $G_{x}$ can be constructed via a finite induction
(see also \prettyref{fig:PartitionAlgo} for an illustration) by subdividing
`$x$-\emph{bad cubes}' into $2^{d}$ subcubes and picking in each
inductive step the `$x$-\emph{good cubes}'.

\subsection*{Adaptive~Approximation~Algorithm.}

\emph{For $x>1/\J\left(\Q\right)$ we initialise our induction by
setting $\mathcal{B}_{0}\coloneqq\left\{ \Q\right\} $ and $\mathcal{G}_{0}=\emptyset$.
Now, suppose the set of '$x$-bad cubes' $\mathcal{B}_{n}\subset\D_{n}$
and '$x$-good cubes' $\mathcal{G}_{n}\subset\D_{0}\cup\cdots\cup\D_{n}$
of generation $n\in\N_{0}$ are given. Then we set 
\begin{align*}
\mathcal{B}_{n+1} & \coloneqq\left\{ Q\in\D_{n+1}:\exists Q'\in\mathcal{B}_{n}:Q\subset Q',\J\left(Q\right)\geq1/x\right\} \:\text{and }\\
\mathcal{G}_{n+1} & \coloneqq\left\{ Q\in\D_{n+1}:\exists Q'\in\mathcal{B}_{n}:Q\subset Q',\J\left(Q\right)<1/x\right\} \cup\mathcal{G}_{n}.
\end{align*}
 Since $\J$ is uniformly vanishing, this procedure terminates after
say $m_{x}\in\N$ steps with $\mathcal{B}_{m_{x}+1}=\emptyset$ and
we return $\mathcal{G}_{m_{x}+1}$.}

\begin{figure}
\includegraphics[width=0.5\textwidth]{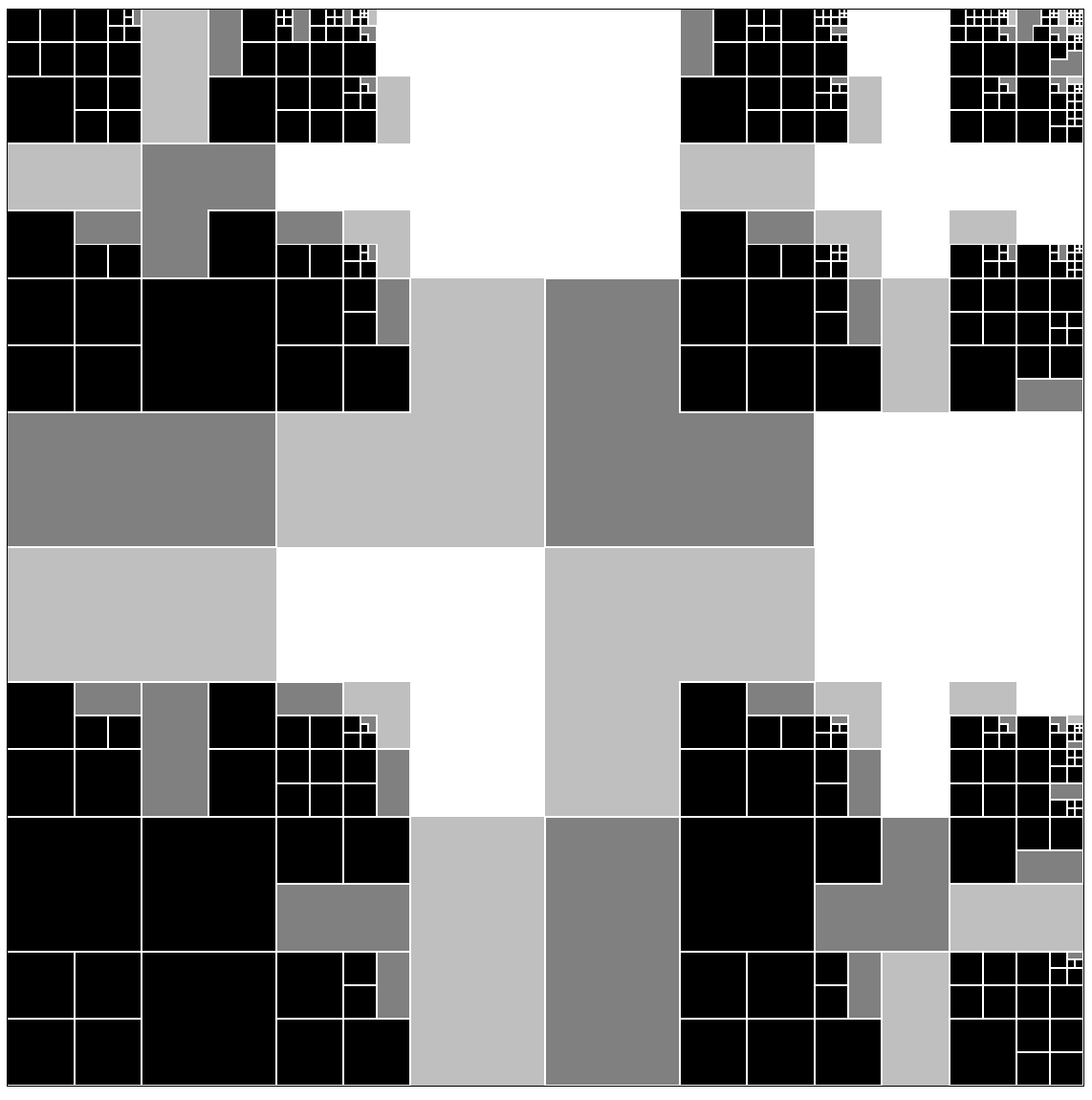}

\caption{\label{fig:PartitionAlgo}Illustration of the adaptive approximation
algorithm for $\J(Q)\protect\coloneqq\left(\nu\varotimes\nu\right)(Q)\left(\Lambda(Q)\right)^{2}$,
$Q\in\mathcal{D}^{*}\protect\coloneqq\left\{ Q\in\mathcal{D}:\J\left(Q\right)>0\right\} $,
$d=2$, where $\nu$ denotes the $\left(0.1,0.9\right)$-Cantor measure
supported on the triadic Cantor set, i.\,e\@. the self similar measure
generated by the IFS $S_{1}:x\protect\mapsto x/3$, $S_{2}:x\protect\mapsto x/3+2/3$
and probability weights $p_{1}=0.1$ and $p_{2}=0.9$ (see \cite{MR625600}).
Here, the light grey cubes belong to $G_{10^{-3}}$, the grey cubes
belong to $G_{10^{-4}}$ and the black cubes belong to $G_{10^{-7}}$.
Of course, the darker cubes overlay the lighter ones.}
\end{figure}

~

The following lemma shows that for $\DD=\D$ the above algorithm indeed
recovers the $x$-good partition $G_{x}$ and that this set solves
an optimisation problem. The proofs for the following lemmas are postponed
to the last section, which is also devoted to the proofs of the main
results.
\begin{lem}
\label{lem:GxOptimal}For $\J:\D\to\R_{\geq0}$, $x>1/\J(\Q)$ and
with the notation given in the Adaptive Approximation Algorithm we
have 
\[
G_{x}=\mathcal{G}_{m_{x}+1}
\]
and this set solves the following optimisation problem: For $\tilde{P}$
from the set $\Pi$ of partitions of $\Q$ with elements from $\D$,
we have 
\[
\card\left(\tilde{P}\right)=\inf\left\{ \card\left(P\right):P\in\Pi,\J\left(P\right)<1/x\right\} \iff\tilde{P}=G_{x}.
\]
\end{lem}

Similarly, also the dual problem is well known in the literature and
closely connected also to the study of quantization dimensions (\cite{KN22b,KN2024QDnegativeorder}).
\begin{lem}
\label{lem:DualClassics}For $\DD=\D$ and $\tilde{\Pi}_{n}$ denoting
the set of partitions of $\Q$ with elements from $\D$ and cardinality
not exceeding $n\in\N$, we have
\[
\gamma_{n}=\inf_{P\in\tilde{\Pi}_{n}}\J\left(P\right).
\]
\end{lem}

With this connection it will turn out that our results (see \prettyref{thm:ImproveBirmanSolomyak})
improve some classical work in this respect, e.\,g\@. \cite[Theorem 2.1]{MR0217487}.

\subsection{\emph{$\J$-}partition functions\label{subsec:-Partition-functions}}

Next, let us turn to the concept of partition functions, which in
a certain extent is borrowed from the thermodynamic formalism. Our
most powerful auxiliary object is the $\J$-\emph{partition function},\emph{
}for $q\in\R_{\geq0}$, given by
\begin{align}
\GL\left(q\right) & \coloneqq\GL_{\J}(q)\coloneqq\limsup_{n\rightarrow\infty}\GL_{n}\left(q\right)\:\text{with\, }\GL_{n}\left(q\right)\coloneqq\GL_{\J,n}\left(q\right)\coloneqq\frac{1}{\log\left(2^{n}\right)}\log\sum_{Q\in\DD_{n}}\J\left(Q\right)^{q}.\label{eq:DefGL}
\end{align}
\begin{figure}
\center{\begin{tikzpicture}[scale=0.9, every node/.style={transform shape},line cap=round,line join=round,>=triangle 45,x=1cm,y=1cm] \begin{axis}[ x=2.7cm,y=2.7cm, axis lines=middle, axis line style={very thick},ymajorgrids=false, xmajorgrids=false, grid style={thick,densely dotted,black!20}, xlabel= {$q$}, ylabel= {$\GL (q)$}, xmin=-0.49 , xmax=3.5 , ymin=-0.3, ymax=2.5,x tick style={color=black}, xtick={0,1,2,2.48,3},xticklabels = {0,1,2,$\q$,3},  ytick={0,1,2},yticklabels = {0,$\GL(1)=1$,$\GL(0)=2$}] \clip(-0.5,-0.3) rectangle (4,4); 
\draw[line width=1pt,smooth,samples=180,domain=-0.3:3.4] plot(\x,{log10(0.08^((\x))+0.2^((\x))+0.36^((\x))+0.36^((\x)))/log10(2)+\x}); 
\draw [line width=01pt,dotted, domain=-0.05 :4.4] plot(\x,{(((log10(0.36))/(log10(2))+1)*(\x-1)+1)});
\draw [line width=01pt,dashed, domain=-0.05 :4.4] plot(\x,{(1-\x)+1)});
 
\node[circle,draw] (c) at (2.48 ,0 ){\,};

\draw [line width=.7pt,dotted, gray] (1 ,0.)--(1,1); 
\draw [line width=.7pt,dotted, gray] (0  ,1.0 )-- (1,1);
\end{axis} 
\end{tikzpicture}}

\caption{\label{fig:Moment-generating-function}A typical partition function
$\GL$ with $\GL\left(0\right)=2$, $\GL\left(1\right)=1$ and $\dim_{\infty}\left(\J\right)>0$.
Natural bounds for $\overline{h}=\q>1$ in this setting are the zeros
of the dashed line $q\protect\mapsto-q\left(\GL\left(0\right)-\GL\left(1\right)\right)+\GL\left(0\right)$
and the dotted line $q\protect\mapsto\left(1-q\right)\dim_{\infty}\left(\J\right)+\GL\left(1\right)$
as given in \prettyref{prop:Geometric_Bounsd}.}
\end{figure}
 Note that we use the convention $0^{0}=0$, that is for $q=0$ we
neglect the summands with $\J\left(Q\right)=0$ in the definition
of $\GL_{n}$. The function $\GL$ is convex as a limit superior of
convex functions. Further, for $\J:\DD\to\R_{\geq0}$ we let $\J^{*}$
denote the restriction of $\J$ to $\DD^{*}\coloneqq\left\{ Q\in\DD:\J\left(Q\right)>0\right\} $
and we observe
\[
\GL_{\J}=\GL_{\J^{*}}.
\]
 We call
\begin{equation}
\dim_{\infty}\left(\J\right)\coloneqq\liminf_{n\to\infty}\frac{\log\left(\J\left(\DD_{n}\right)\right)}{-\log\left(2^{n}\right)}\label{eq:Def_Dim_infty}
\end{equation}
the \emph{$\infty$-dimension of} $\J$, which we often assume to
be strictly positive and in turn leads to $\J$ being uniformly vanishing
(see \prettyref{lem:Unifrom_Decreasing}).

To exclude trivial cases, we will always assume that there exist $a>0$
and $b\in\R$ such that 
\begin{equation}
\GL_{n}\left(a\right)\geq b\label{eq:uniform_proper_convex}
\end{equation}
for all $n\in\N$ large enough; in particular $\GL$ is a proper convex
function. All relevant examples mentioned above fulfil this condition.

Since the maximal asymptotic direction $\lim_{q\to\infty}\GL\left(q\right)/q$
of $\GL$ coincides with $-\dim_{\infty}\left(\J\right)$, $\dim_{\infty}\left(\J\right)>0$
implies that the \emph{critical exponent} 
\[
\kappa\coloneqq\kappa_{\J}\coloneqq\inf\left\{ q\geq0:\sum_{Q\in\DD}\J\left(Q\right)^{q}<\infty\right\} \;\text{coincides with }\;\q\coloneqq\q_{\J}\coloneqq\inf\left\{ q\geq0:\GL\left(q\right)<0\right\} .
\]
If $0<\q<\infty$, then $\q$ is the unique zero of $\GL$ and $\q=\kappa$;
in general, we have $\q\leq\kappa$ (cf\@. \prettyref{lem:Dim00Inequality}).
Note that also $\q=\limsup_{n\to\infty}q_{n}$, where $q_{n}$ denotes
the unique zero of $\GL_{n}$. Let us also write 
\begin{equation}
\underline{\q}\coloneqq\liminf_{n\to\infty}q_{n}.\label{eq:Def_Of_underline_q}
\end{equation}
This quantity will be relevant for upper and lower bounds on the lower
optimised coarse multifractal dimension introduced in the next section
(see \prettyref{prop:LowerBoundLiminfFJ}).

Since $\GL$ does not change when we replace $\DD$ by $\DD^{*}$,
we conclude that $\q_{\J}=\q_{\J^{*}}$.

\subsubsection{Coarse multifractal dimensions}

For the lower bounds, we use a concept closely connected to the coarse
multifractal analysis (see e.\,g\@. \cite{MR3236784}) . For all
$n\in\N$ and $\alpha>0$, we define 
\begin{equation}
N_{\alpha}\left(n\right)\coloneqq\card\mathcal{N}_{\alpha}\left(n\right),\quad\mathcal{N}_{\alpha}\left(n\right)\coloneqq\left\{ Q\in\DD_{n}:\J\left(Q\right)\geq2^{-\alpha n}\right\} ,\label{eq:Def:N_=00005Calpha(n)}
\end{equation}
(for an illustration of $\mathcal{N}_{\alpha}\left(n\right)$ for
an concrete example with optimal $\alpha$, we refer to \prettyref{fig:PartitionAlgoAndMultiFractal})
and set 
\begin{figure}
\includegraphics[width=0.5\textwidth]{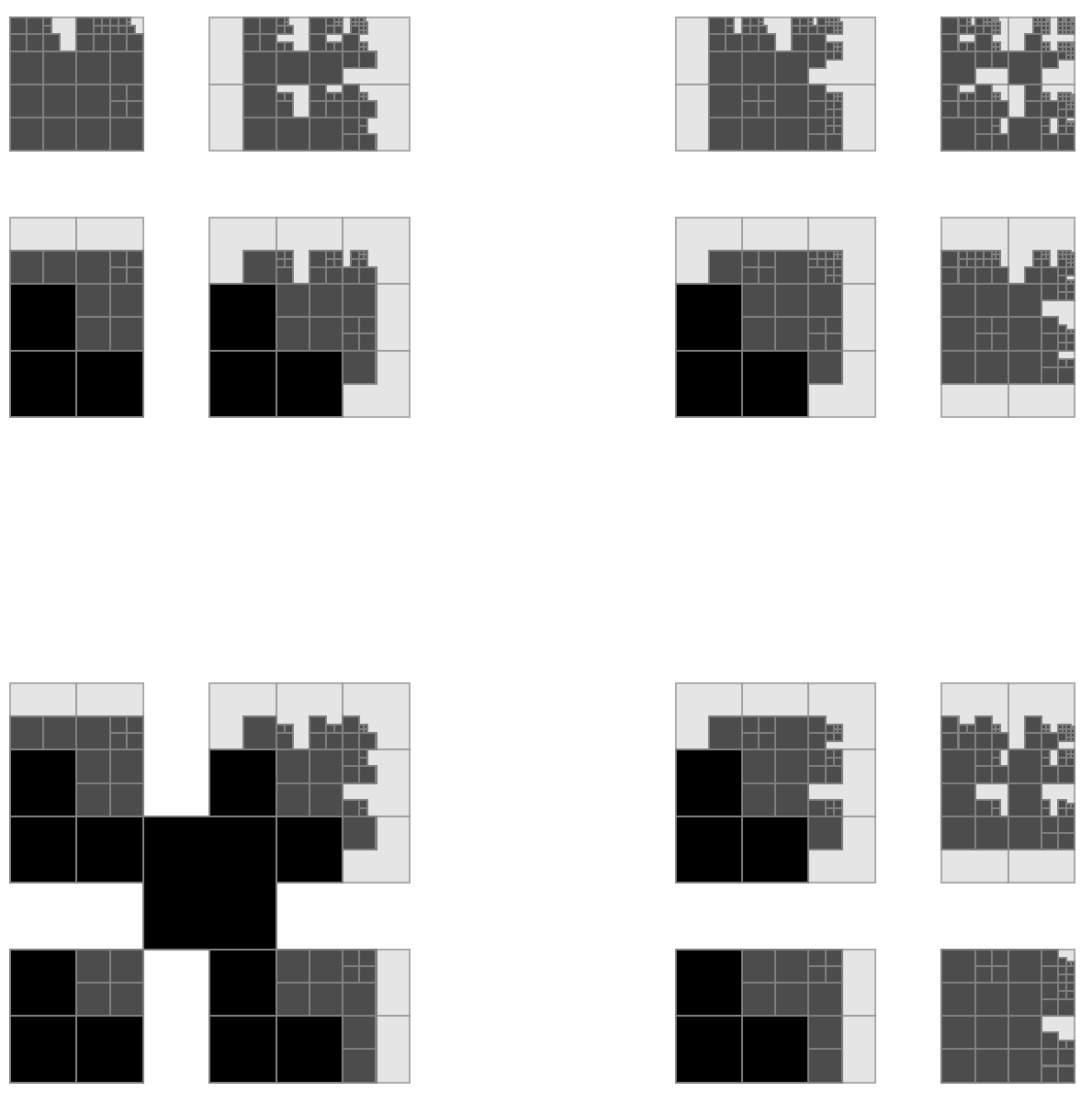}

\caption{\label{fig:PartitionAlgoAndMultiFractal} Illustration for the same
example as in \prettyref{fig:PartitionAlgo} of the cubes $\mathcal{N}_{\alpha}\left(n\right)$
(light grey) and $G_{2^{-\alpha n}}$ (black, or dark grey if covered
by an element of $\mathcal{N}_{\alpha}\left(n\right)$) with $n=4$,
$\alpha=5.734$.}
\end{figure}
\begin{equation}
\overline{F}\left(\alpha\right)\coloneqq\limsup_{n\to\infty}\frac{\log^{+}\left(N_{\alpha}\left(n\right)\right)}{\log\left(2^{n}\right)}\;\text{and }\;\underline{F}\left(\alpha\right)\coloneqq\liminf_{n\to\infty}\frac{\log^{+}\left(N_{\alpha}\left(n\right)\right)}{\log\left(2^{n}\right)},\label{eq:Def_F_(alpha)}
\end{equation}
with $\log^{+}(x)\coloneqq\max\left\{ 0,\log(x)\right\} $, $x\geq0$.
We refer to the quantities
\begin{equation}
\overline{F}\coloneqq\overline{F}_{\J}\coloneqq\sup_{\alpha>0}\frac{\overline{F}\left(\alpha\right)}{\alpha}\quad\text{and }\quad\underline{F}\coloneqq\underline{F}_{\J}\coloneqq\sup_{\alpha>0}\frac{\underline{F}\left(\alpha\right)}{\alpha}\label{eq:Def_F_}
\end{equation}
as the \emph{upper}, resp\@. \emph{lower, optimised coarse multifractal
dimension} with respect to $\J$.

At this point we would like to point out that the reciprocal quantities
closely related to the concept of $n$-widths have already been considered
in the work of Birman and Solomjak \cite{MR0217487,MR562305}; in
\prettyref{subsec:PartitionImprovemetBirmanSolomyak} we show that
our formalism gives improved estimates on the asymptotic rates obtained
by Birman and Solomjak.

\subsection{Main results\label{subsec:Main-results}}

Our main result are stated for $\J:\DD\to\R_{\geq0}$ as given above;
all proofs for this section are postponed to \prettyref{sec:Partition-functions}.
\begin{thm}
\label{thm:MainResult} If $\dim_{\infty}\left(\J\right)>0$, then
\[
\underline{F}\leq\underline{h}\leq\overline{h}=\q=\kappa=\overline{F}.
\]
\end{thm}

\begin{rem}
From the definition it is clear that for $\mathcal{T}\subset\DD$,
all quantities above are monotone in the sense that $\underline{F}$,
$\underline{h}$, etc., which are defined with respect to $\J|_{\mathcal{T}}$
do not exceed $\underline{F}$, $\underline{h}$, etc., which are
defined with respect to $\J$. Further, for the restriction on $\DD^{*}$,
we have $\overline{h}_{\J}=\overline{h}_{\J^{*}}$, which can be seen
in two ways: either use $\GL_{\J}=\GL_{\J^{*}}$ or, alternatively,
$\overline{F}_{\J}=\overline{F}_{\J^{*}}$ and \prettyref{thm:MainResult}.
Also note that $\underline{F}_{\J}=\underline{F}_{\J^{*}}$.

In our proofs we will see that if $\J$ is uniformly vanishing and
allowing $\dim_{\infty}\left(\J\right)=0$, we still have 
\[
\overline{h}\leq\kappa\leq\q.
\]
\end{rem}

\begin{cor}
\label{cor:special_nu^s}Let $\nu$ be a finite Borel measure on $\Q$,
we consider $\J_{s}:\D\to\R_{\geq0},Q\mapsto\left(\nu\left(Q\right)\right)^{s}$
for some $s>0$ and such that $\dim_{\infty}\left(\J_{1}\right)>0$.
Then we have 
\[
\underline{h}_{\J_{s}}=\overline{h}_{\J_{s}}=\q_{\J_{s}}=1/s.
\]
\end{cor}

\begin{proof}
We readily see that for $x>1/\J(\Q)$ and for $Q\in G_{x}$ we have
$\nu(Q)x^{1/s}<1$. Therefore, 
\[
x^{1/s}\nu(\Q)=x^{1/s}\sum_{Q\in G_{x}}\J(Q)^{1/s}<\sum_{Q\in G_{x}}1=\card\left(G_{x}\right),
\]
 proving $1/s\leq\underline{h}_{\J_{s}}$. Also, $\q_{\J_{s}}=1/s$
is immediate. Hence, the equalities follow from \prettyref{thm:MainResult}.
\end{proof}
\begin{rem}
In \cite[Section 3.2]{MR4484835}, the set function $\nu^{s}$ is
crucial to estimate the eigenvalues of Birman–Schwinger operators.This
work follows a different approach; instead of dyadic cubes, aligned
cubes contained in $\Q$ have been considered. This improves the upper
estimate, in the sense that there exists a constant $c>0$ such that
for $x$ large, $M_{\J_{s}}\left(x\right)\leq cx^{1/s}$.
\end{rem}

\begin{thm}
\label{thm:ImproveBirmanSolomyak} Assuming $\dim_{\infty}\left(\J\right)>0$,
we have
\[
\frac{-1}{\underline{F}}\leq\frac{-1}{\underline{h}}=\underline{\alpha}\leq\overline{\alpha}=\frac{-1}{\overline{h}}=\frac{-1}{\q}.
\]
\end{thm}

In \prettyref{subsec:PartitionImprovemetBirmanSolomyak}, we give
the proof and further discussions in the context of the classical
work \cite{MR0217487,Borzov1971}.

\subsubsection{Fractal-Geometric bounds}

We define the \emph{support of} $\J$ to be 
\begin{align*}
\supp\left(\J\right) & \coloneqq\bigcap_{k\in\N}\overline{\bigcup_{n\geq k}\left\{ Q:Q\in\DD_{n},\J\left(Q\right)>0\right\} }.
\end{align*}
Note that if $\J$ is given by a measure $\nu$ restricted to the
dyadic cubes $\D$, then our definition of the support coincides with
the usual definition of the support of measures. We write 
\[
\overline{\dim}_{M}\left(A\right)\coloneqq\limsup_{n\rightarrow\infty}\frac{\log\left(\card\left(\left\{ Q\in\D_{n}:Q\cap A\neq\emptyset\right\} \right)\right)}{\log\left(2^{n}\right)}\in\left[0,d\right]
\]
for the\emph{ upper Minkowski dimension} of the set $A\subset\Q$.
Slightly abusing notation, we also write 
\begin{align*}
\overline{\dim}_{M}\left(\J\right) & \coloneqq\overline{\dim}_{M}\left(\supp\left(\J\right)\right).
\end{align*}

In several applications of our results, the value of $\GL\left(1\right)$
is easily accessible (see e.\,g\@. \cite{KN2022b}), the following
expressions provide convenient bounds. For an illustrating example
see \prettyref{fig:Moment-generating-function}.
\begin{prop}
\label{prop:Geometric_Bounsd}If $1\leq\q<\infty$, then 
\[
\frac{\GL\left(0\right)}{\GL\left(0\right)-\GL\left(1\right)}\leq\q\leq\frac{\dim_{\infty}\left(\J\right)+\GL\left(1\right)}{\dim_{\infty}\left(\J\right)}\leq\frac{\tau(0)}{\dim_{\infty}\left(\J\right)}\leq\frac{\overline{\dim}_{M}\left(\J\right)}{\dim_{\infty}\left(\J\right)}\leq\frac{d}{\dim_{\infty}\left(\J\right)},
\]
and if $\q\leq1$, then 
\[
\frac{\dim_{\infty}\left(\J\right)+\GL\left(1\right)}{\dim_{\infty}\left(\J\right)}\leq{\displaystyle \q\leq\frac{\GL(0)}{\GL(0)-\GL\left(1\right)}\leq\frac{\overline{\dim}_{M}\left(\J\right)}{\overline{\dim}_{M}\left(\J\right)-\GL\left(1\right)}.}
\]
\end{prop}

\subsubsection{Regularity results}
\begin{defn}
Assuming $\dim_{\infty}\left(\J\right)>0$, we define two notions
of regularity.
\begin{enumerate}
\item We call $\J$\emph{ multifractal-regular (MF-regular)} if $\underline{F}=\overline{F}$.
\item We call $\J$ \emph{partition function regular (PF-regular)} if
\begin{itemize}
\item $\GL\left(q\right)=\liminf_{n\to\infty}\GL_{n}\left(q\right)$ for
$q\in\left(\q-\varepsilon,\q\right)$, for some $\varepsilon>0$,
or
\item $\GL\left(\q\right)=\liminf_{n\to\infty}\GL_{n}\left(\q\right)$ and
$\GL$ is differentiable at $\q$.
\end{itemize}
\end{enumerate}
\end{defn}

\begin{rem}
The above theorem and the notion of regularity give rise to the following
list of observations:
\end{rem}

\begin{enumerate}
\item An easy calculation shows that 
\[
\underline{F}\leq\inf\left\{ q>0:\liminf_{n\to\infty}\GL_{n}\left(q\right)<0\right\} =\underline{\q}\leq\q=\overline{F}.
\]
From this it follows that MF-regularity implies that $\GL$ exists
as a limit in $\q$.
\item If $\J$ is MF-regular, then equality holds everywhere in the chain
of inequalities \prettyref{thm:MainResult}.
\end{enumerate}
The following theorem shows that the $\J$-partition function is in
many situations a valuable auxiliary concept to determine the exact
value of the $\J$-partition entropy.
\begin{thm}
\label{thm:LqRegularImpliesRegular}Assume $\dim_{\infty}(\J)>0$.
If $\J$ is PF-regular, then it is MF-regular. If $\J$ is MF-regular,
then the $\J$-partition entropy $h$ exists with $h=\q=\ensuremath{\kappa}=F$.
\end{thm}

\begin{rem}
\label{rem:LowerBdLower}The above result is optimal in the sense
that there is an example of a measure $\nu$ (derived in the context
of for Kre\u{\i}n–Feller operators in dimension $d=1$ in \autocite{KN2022})
such that $\J_{\nu}\coloneqq\nu$ is not PF-regular and for which
$\overline{h}>\underline{h}$. It should be noted that PF-regularity
is often easily accessible if the spectral partition function is essentially
given by the $L^{q}$-spectrum of an underlying measure.

Recall the definition of $\underline{\q}$ in \prettyref{eq:Def_Of_underline_q}.
We have seen that $\underline{F}\leq\underline{\q}$ and one could
hope for equality in general. However, the lower bound is considerably
more challenging to estimate, and we are able to make the following
observation.
\end{rem}

\begin{prop}
\label{prop:LowerBoundLiminfFJ}Assuming that $\dim_{\infty}(\J)>0$,
let the convex function 
\[
c:q\mapsto\limsup_{n\to\infty}\tau_{\J,n}\left(q+q_{n}\right)
\]
be given, where $q_{n}$ denotes the only zero of $\GL_{n}$. Then
for $a\leq b<0$ such that the subdifferential $\partial c(0)$ of
$c$ in $0$ is equal to $\left[a,b\right]$ we have 
\[
\frac{b}{a}\,\underline{\q}\leq\underline{F}\leq\underline{\q}.
\]
\begin{rem}
This means that differentiability of $c$ in $0$ implies $\underline{F}=\underline{\q}$.
\end{rem}

\end{prop}

\begin{rem}
See \prettyref{exa:LowerBound} for an example, where the lower bound
in \prettyref{prop:LowerBoundLiminfFJ} is realised, i.\,e\@. $\nicefrac{b}{a}\;\underline{\q}=\underline{F}$.
\end{rem}

\subsection{Possible applications\label{subsec:Possible-applications}}

This paper is partially based on the second author's PhD thesis \cite{elib_6573}.

Let $\nu$ be a Borel measure on $\Q$ . A classical example for $\J$
would be $\nu$ restricted to $\D$. In \prettyref{cor:special_nu^s}
we will provide an illustrating example for $\J_{\nu,s}\left(Q\right)\coloneqq\nu\left(Q\right)^{s}$,
$Q\in\D$, $s>0$, that plays a crucial role in the context of spectral
asymptotics \cite{MR4331823,MR4484835}. In \cite{KN2022b}, $a,b\in\R$,
$b>0$, we studied
\begin{equation}
\J_{\nu,a,b}\left(Q\right)\coloneqq\begin{cases}
\sup\left\{ \nu\left(\widetilde{Q}\right)^{b}\left|\log\left(\Lambda\left(\widetilde{Q}\right)\right)\right|:\widetilde{Q}\in\mathcal{\D}\left(Q\right)\right\} , & a=0,\\
\sup\left\{ \nu\left(\widetilde{Q}\right)^{b}\Lambda\left(\widetilde{Q}\right)^{a}:\widetilde{Q}\in\D\left(Q\right)\right\} , & a\neq0,
\end{cases}\label{eq:J_nu,a,b}
\end{equation}
with $\mathcal{D}\left(Q\right)\coloneqq\left\{ Q'\in\mathcal{D}:Q'\subset Q\right\} $,
$Q\in\D$. Note that for $a>0$ this definition reduces to $\J_{\nu,a,b}\left(Q\right)=\nu\left(Q\right)^{b}\Lambda\left(Q\right)^{a}$.
For an appropriate choice of $a,b$ the set function $\J_{\nu,a,b}$
naturally arises in the optimal embedding constant for the embedding
of the standard Sobolev space $H^{1,2}$ in $L_{\nu}^{2}$. For $\Q\subset\R^{d}$
and $t\geq2$, we were particularly interested in $\J_{\nu,t}\coloneqq\J_{\nu,2/d-1,2/t}$
to investigate the spectral asymptotic of Kre{\u{\i}}n–Feller operators.
We note that the general parameter $a,b$ has also been shown to be
useful when considering \emph{polyharmonic operators} in higher dimensions
or \emph{approximation order} with respect to \emph{Kolmogorov, Gel\textquotesingle fand,}
or \emph{linear widths} as elaborated in \cite{MR4444736,KesseboehmerWiegmann}.
In these works, the deep connection to the original ideas of entropy
numbers introduced by Kolmogorov also becomes apparent. In \cite{KN22b,KN2024QDnegativeorder}
we address the quantization problem, that is the speed of approximation
of a compactly supported Borel probability measure by finitely supported
measures (see \cite{MR1764176} for an exposition), by adapting the
methods developed in \prettyref{sec:OptimalPartitions} and \prettyref{sec:Coarse-multifractal-analysis}
to $\J_{\nu,r,1}$ with $r>-\dim_{\infty}\left(\nu\right)$ to identify
the upper \emph{quantization dimension} \emph{of order} $r$ of $\nu$
with its \emph{Rényi dimension}.

\section{Basic properties of the partition function \label{sec:Partition-functions}}

Recall the definition in \prettyref{eq:DefGL} of the\emph{ }$\J$-partition
function $\GL$ as well as the critical values $\q$ and $\kappa$,
for which we give further observations: One easily checks that $\GL$
is \emph{scale invariant} in the sense that for $c>0$, we have $\GL_{c\J}=\GL_{\J}$.
\begin{lem}
\label{lem:GL(0)=00003DDim_M} We always have $\GL\left(0\right)\leq\overline{\dim}_{M}\left(\J\right)$,
and if $\DD=\D$, then $\GL\left(0\right)=\overline{\dim}_{M}\left(\J\right)$.
\end{lem}

\begin{proof}
We first show that for $Q\in\DD_{n}$ with $\J\left(Q\right)>0$,
we have $\overline{Q}\cap\supp\left(\J\right)\neq\emptyset$. Indeed,
since $\J$ is locally non-vanishing there exists a subsequence $(n_{k})$
with $Q_{n_{k}}\in\DD_{n_{k}}$, $\J\left(Q_{n_{k}}\right)>0$ and
$Q_{n_{k}}\subset Q_{n_{k-1}}\subset Q.$ Since $\left(\overline{Q}_{n_{k}}\right)_{k}$
is a nested sequence of non-empty compact subsets of $\overline{Q}$
we have $\emptyset\neq\bigcap_{k\in\N}\overline{Q}_{n_{k}}\subset\supp\left(\J\right)\cap\overline{Q}$.
Therefore, 
\begin{align*}
\card\left\{ Q\in\DD_{n}:\J\left(Q\right)>0\right\}  & \leq\card\left\{ Q\in\DD_{n}:\overline{Q}\cap\supp\left(\J\right)\neq\emptyset\right\} \\
 & \leq3^{d}\card\left\{ Q\in\D_{n}:Q\cap\supp\left(\J\right)\neq\emptyset\right\} 
\end{align*}
implying $\GL\left(0\right)\leq\overline{\dim}_{M}\left(\J\right)$.

Now, assume $\DD_{n}=\D_{n}$. We observe that if $Q\in\D_{n},Q\cap\supp\left(\J\right)\neq\emptyset$,
then there exists $Q'\in\ \D_{n}$ with $\overline{Q'}\cap\overline{Q}\neq\emptyset$
and $\J(Q')>0$. This can be seen as follows: For $x\in Q\cap\supp\left(\J\right)$
there exists a subsequence $(n_{k})$ such that $x\in\overline{Q}_{n_{k}}$,
$Q_{n_{k}}\in\mathcal{\DD}_{n_{k}}$ and $\J(Q_{n_{k}})>0$. For $k\in\N$
such that $n_{k}\geq n$ there exists exactly one with $Q_{n_{k}}\subset Q'$.
Now, $x\in\overline{Q}_{n_{k}}\subset\overline{Q'}$, implying $\overline{Q'}\cap\overline{Q}\neq\emptyset$
and since $\J$ is monotone, we have $\J(Q')>0$. Furthermore, for
each $Q\in\DD_{n}$, we have $\card\left\{ Q''\in\D_{n}:\overline{Q''}\cap\overline{Q}\neq\emptyset\right\} \leq3^{d}$.
Combining these two observations, we obtain
\begin{align*}
\card\left\{ Q\in\D_{n}:Q\cap\supp\left(\J\right)\neq\emptyset\right\}  & \leq\card\left\{ Q\in\D_{n}:\exists Q'\in\D_{n},\overline{Q'}\cap\overline{Q}\neq\emptyset,\J(Q')>0\right\} \\
 & \leq3^{d}\card\left\{ Q\in\D_{n}:\J\left(Q\right)>0\right\} ,
\end{align*}
implying $\GL\left(0\right)\geq\overline{\dim}_{M}\left(\J\right)$.
\end{proof}
The definition of $\dim_{\infty}\left(\J\right)>0$ in \prettyref{eq:Def_Dim_infty}
immediately gives the following lemma.
\begin{lem}
\label{lem:Unifrom_Decreasing}If $\dim_{\infty}\left(\J\right)>0$,
then $\J$ is uniformly vanishing.
\end{lem}

\begin{lem}
\label{lem:uniformEstimatefortau_n}Under our standing assumption
with $a$ and $b$ as given in \prettyref{eq:uniform_proper_convex},
$L\coloneqq(b-d)/a<0$, for all $n$ large enough and $q\geq0$, we
have 
\[
b+qL\leq\GL_{n}\left(q\right).
\]
In particular, $-\infty<\liminf_{n\rightarrow\infty}\GL_{n}\left(q\right)$\textbf{
}and\textbf{ $\dim_{\infty}\left(\J\right)\leq-L$}
\end{lem}

\begin{proof}
By our assumptions we have $\dim_{\infty}\left(\J\right)>0$, therefore,
for $n$ large, $\GL_{n}$ is monotone decreasing and also $b\leq\GL_{n}(a)$.
By definition, we have $\GL_{n}(0)\leq d$ for all $n\in\N$ and the
convexity of $\GL_{n}$ implies for all $q\in[0,a]$
\[
\GL_{n}\left(q\right)\leq\GL_{n}(0)+\frac{q\left(\GL_{n}(a)-\GL_{n}(0)\right)}{a}.
\]
On the other hand, for $q>a$, the convexity of $\GL_{n}$ implies
\[
\frac{\left(\GL_{n}(a)-\GL_{n}(0)\right)}{a}\leq\frac{\left(\GL_{n}\left(q\right)-\GL_{n}(0)\right)}{q}
\]
and consequently, 
\begin{align*}
b+q(b-d)/a & \leq\GL_{n}(0)+\frac{q\left(\GL_{n}(a)-\GL_{n}(0)\right)}{a}\\
 & \leq\GL_{n}(0)+\frac{q\left(\GL_{n}\left(q\right)-\GL_{n}(0)\right)}{q}=\GL_{n}\left(q\right).
\end{align*}
Since $\GL_{n}$ is decreasing with $0\leq\GL_{n}(0)\leq d$ and $\GL_{n}(a)\geq b$,
we obtain for all $q\in[0,a]$
\[
b+q(b-d)/a\leq b\leq\GL_{n}(a)\leq\GL_{n}\left(q\right).
\]
\end{proof}
In the following lemma we use the convention $-\infty\cdot0=0$.
\begin{lem}
\label{lem:Dim00Inequality}For $q\geq0$, we have 
\begin{align}
-\dim_{\infty}\left(\J\right)q\leq\GL\left(q\right) & \leq\GL(0)-\dim_{\infty}\left(\J\right)q\label{eq:InequalitiesTauP}\\
 & \leq\overline{\dim}_{M}\left(\J\right)-\dim_{\infty}\left(\J\right)q.\nonumber 
\end{align}
Furthermore, 
\[
\dim_{\infty}\left(\J\right)>0\iff\q<\infty\implies\kappa=\q.
\]
\end{lem}

\begin{proof}
The first claim follows from the following simple inequalities
\[
q\log\left(\J\left(\mathcal{\DD}_{n}\right)\right)\leq\log\left(\sum_{Q\in\mathcal{\DD}_{n}}\J\left(Q\right)^{q}\right)\leq\log\left(\sum_{Q\in\mathcal{\DD}_{n},\J\left(Q\right)>0}1\right)+q\log\left(\J\left(\mathcal{\DD}_{n}\right)\right).
\]
Now, assume $\q<\infty$. It follows there exists $q>0$ such that
$\GL\left(q\right)<0$. Consequently, we obtain from \prettyref{eq:InequalitiesTauP}
$-\dim_{\infty}\left(\J\right)q\leq\GL\left(q\right)<0$, which gives
$\dim_{\infty}\left(\J\right)>0.$ Reversely, suppose $\dim_{\infty}(\J)>0$.
In the case $\dim_{\infty}\left(\J\right)=\infty$, using \prettyref{eq:InequalitiesTauP},
we have $\q=0$ due to $\GL\left(q\right)=-\infty$ for $q>0$. Now,
let us consider the case $0<\dim_{\infty}\left(\J\right)<\infty$.
Then it follows from \prettyref{eq:InequalitiesTauP} that $\GL\left(q\right)<0$
for all $q>\GL(0)/\dim_{\infty}\left(\J\right)$ which proves the
implication.

Now, assume $\q<\infty$. Then we have $\GL\left(q\right)<0$ for
all $q>\q$, and therefore, for every $\varepsilon>0$ with $\GL\left(q\right)<-\varepsilon<0$
and $n$ large enough, we obtain $\sum_{Q\in\mathcal{\DD}_{n}}\J\left(Q\right)^{q}\leq2^{-n\varepsilon},$
implying $\sum_{Q\in\mathcal{\DD}}\J\left(Q\right)^{q}<\infty.$ This
shows $\inf\left\{ q\geq0:\sum_{Q\in\mathcal{\DD}}\J\left(Q\right)^{q}<\infty\right\} \leq\q$.
For the reversed inequality we note that if $\q=0$, then the claimed
equality is clear. If, on the other hand, $\q>0$, then we necessarily
have $\dim_{\infty}\left(\J\right)<\infty$. Since, $\GL$ is decreasing,
convex and proper (see \prettyref{lem:strictlyDecreaing} below),
it follows that $\q$ is a zero of $\GL$ and for all $0<q<\q$ we
have $0<\GL\left(q\right)$. This implies that for every $0<\delta<\GL\left(q\right)$,
there is a subsequence $(n_{k})$ such that 
\[
2^{n_{k}\delta}\leq\sum_{Q\in\mathcal{\DD}_{n_{k}}}\J\left(Q\right)^{q}\:\text{implying \,}\infty=\sum_{k\in\N}\sum_{Q\in\mathcal{\DD}_{n_{k}}}\J\left(Q\right)^{q}\leq\sum_{Q\in\mathcal{\DD}}\J\left(Q\right)^{q}.
\]
Consequently, $\q\leq\inf\left\{ q\geq0:\sum_{Q\in\mathcal{\DD}}\J\left(Q\right)^{q}<\infty\right\} $.
\end{proof}
\begin{rem}
Note that in the case $\dim_{\infty}\left(\J\right)\leq0$, we deduce
from \prettyref{lem:Dim00Inequality} that $\GL\left(q\right)$ is
non-negative for $q\geq0$, hence $\q=\infty$. However, it is possible
that $\kappa<\infty.$ Indeed, in \cite{KN2022b} we give an example
of a measure $\nu$, where $\kappa_{\J_{\nu}}$ gives the precise
upper bound for the spectral dimension, while $\kappa_{\J_{\nu}}<\q_{\J_{\nu}}=\infty$.
\end{rem}

\begin{lem}
\label{lem:strictlyDecreaing}If $\dim_{\infty}\left(\J\right)\in(0,\infty)$,
then $\GL$ is a strictly decreasing real-valued convex function on
$\R_{\geq0}.$ In particular, if $\q>0$, then $\q$ is the only zero
of $\GL$.
\end{lem}

\begin{proof}
First, note that \prettyref{lem:Dim00Inequality} implies $\GL\left(q\right)\in\R$
for all $q\geq0$ and $\lim_{q\to\infty}\GL\left(q\right)=-\infty.$
Since $\dim_{\infty}\left(\J\right)>0$ it follows from \prettyref{lem:Unifrom_Decreasing}
that for $n$ large and all $Q\in\mathcal{\DD}_{n}$, we have $\J\left(Q\right)<1.$
Hence, $\GL$ is decreasing and as pointwise limit superior of convex
functions again convex. Now, we show that $\GL$ is strictly decreasing.
Assume there exist $0\leq q_{1}<q_{2}$ such that $\GL(q_{1})=\GL(q_{2}).$
Since $\GL$ is decreasing, we obtain $\GL(q_{1})=\GL\left(q\right)$
for all $q\in[q_{1},q_{2}].$ The convexity of $\GL$ implies $\GL\left(q\right)=\GL(q_{1})$
for all $q>q_{1}$ which contradicts $\lim_{q\to\infty}\GL\left(q\right)=-\infty.$
For the second claim note that, since $\GL$ is convex, it follows
that $\GL$ is continuous on $\R_{>0}$. Hence, we obtain $\GL(\q)=0$.
Finally, the uniqueness follows from the fact that $\GL$ is a finite
strictly decreasing function.
\end{proof}

\section{Optimal partitions, partition entropy and optimised coarse multifractal
dimension \label{sec:OptimalPartitions}}

\subsection{Bounds for the partition entropy}

As before, let $\mathsf{\mathscr{\J}:\DD\to\R_{\geq0}}$ be a non-trivial,
monotone, uniformly vanishing, and locally non-vanishing set function.
\begin{prop}
\label{prop:GeneralUpperBounds} For $0<1/x<j_{0}$, the growth rate
of $\card\left(G_{x}\right)$ gives rise to the following inequalities:
\begin{equation}
\overline{F}\leq\overline{h}\leq\kappa\leq\q,\qquad\;\underline{F}\le\underline{h}.\label{eq:MainChainThm}
\end{equation}
\end{prop}

At this stage we would like to point out that in the next section
(\prettyref{prop:LowerBoundUpperSpecDim}), we will show equality
in the second chain of inequalities \prettyref{eq:MainChainThm} using
the coarse multifractal formalism under some mild additional assumptions
on $\J$.
\begin{proof}
Since $\J$ is uniformly vanishing, \prettyref{lem:Dim00Inequality}
gives $\kappa\leq\q$ (where equality holds if $\dim_{\infty}\left(\J\right)>0$,
otherwise $\q=\infty$). Hence, we only have to consider the case
$\kappa<\infty$. Let $0<1/x<j_{0}$. Setting $R_{x}\coloneqq\left\{ Q\in\DD:\J\left(Q\right)\geq1/x\right\} ,$we
note that, on the one hand, for $Q\in G_{x}$ there is exactly one
$Q'\in R_{x}\cap\DD_{\left|\log_{2}\Lambda\left(Q\right)\right|/d-1}$
with $Q\subset Q'$ and, on the other hand, for each $Q'\in R_{x}\cap\DD_{\left|\log_{2}\Lambda\left(Q\right)\right|/d-1}$
there are at most $2^{d}$ elements of $G_{x}\cap\DD_{\left|\log_{2}\Lambda\left(Q\right)\right|/d}$
which are subsets of $Q'$. Hence, $\card\left(G_{x}\cap\DD_{n}\right)\leq2^{d}\card\left(R_{x}\cap\DD_{n-1}\right)$.
For $q>\kappa$ we obtain
\begin{align*}
x^{-q}\card\left(G_{x}\right) & =\sum_{n=1}^{\infty}\sum_{Q\in G_{x}\cap\DD_{n}}x^{-q}\leq2^{d}\sum_{n=1}^{\infty}\sum_{Q\in R_{x}\cap\DD_{n-1}}x^{-q}\\
 & \leq2^{d}\sum_{n=1}^{\infty}\sum_{Q\in R_{x}\cap\DD_{n-1}}\J\left(Q\right)^{q}\leq2^{d}\sum_{n=0}^{\infty}\sum_{Q\in\DD_{n}}\J\left(Q\right)^{q}<\infty.
\end{align*}
This implies
\[
\limsup_{x\to\infty}\frac{\log\left(M\left(x\right)\right)}{\log(x)}\leq q
\]
and letting $q$ tend to $\kappa$ proves $\overline{h}\leq\kappa$.
To prove the first inequality, observe that for $\alpha>0$, $n\in\N$
we have
\begin{equation}
N_{\alpha}\left(n\right)=\card\left\{ Q\in\DD_{n}:\J\left(Q\right)\geq2^{-\alpha n}\right\} \leq\card\left(G_{2^{\alpha n}}\right)=M\left(2^{\alpha n}\right),\label{eq:Nalpha(n)}
\end{equation}
where we used the fact that, since $\J$ is uniformly vanishing and
locally non-vanishing, for each $Q\in\DD_{n}$ with $\J\left(Q\right)\geq2^{-\alpha n}$
there exists at least one $Q'\in\DD\left(Q\right)\cap G_{2^{\alpha n}}$
and this assignment is injective. Taking logarithms, dividing by $\alpha n\log\left(2\right)$,
taking the limit superior with respect to $n$ and then the supremum
over all $\alpha>0$ gives $\overline{F}\leq\overline{h}$.

It remains to prove $\underline{F}\leq\underline{h}$. For fixed $\alpha>0$,
there exists $n_{x}\in\N$ such that $2^{-\left(n_{x}+1\right)\alpha}<1/x\leq2^{-n_{x}\alpha}$
and by \prettyref{eq:Nalpha(n)} we have $N_{\alpha}\left(n_{x}\right)\leq M\left(x\right).$
Therefore, 
\begin{align*}
\liminf_{n\rightarrow\infty}\frac{\log\left(N_{\alpha}\left(n\right)\right)}{\alpha\log\left(2^{n}\right)} & \leq\liminf_{x\rightarrow\infty}\frac{\log\left(N_{\alpha}\left(n_{x}\right)\right)}{\log(x)}\leq\liminf_{x\rightarrow\infty}\frac{\log\left(M\left(x\right)\right)}{\log(x)}=\underline{h}
\end{align*}
and taking the supremum over $\alpha>0$ gives $\underline{F}\leq\underline{h}$.
\end{proof}
\begin{rem}
\label{rem:AdaptiveApproxAlgorithm} We provide a two-dimensional
illustration in \prettyref{fig:PartitionAlgo} of these partitions
$G_{x}$ for three different values of $x>1$ for the particular choice
$\J(Q)=\left(\nu\varotimes\nu\right)(Q)\Lambda(Q)^{2}$, $Q\in\mathcal{D}$,
where $\nu$ denotes the $\left(p,1-p\right)$-Cantor measure supported
on the triadic Cantor set.
\end{rem}

In general, it is difficult to determine an upper bound for the lower
$\J$-partition entropy; the following proposition opens up a feasible
condition which we used \cite{KN2022b} to construct an Kre\u{\i}n–Feller
operator for which the spectral dimension does not exist. To obtain
meaningful bounds in the following theorem, it is important that $\J|_{\DD_{n}}$
does not vary too much on a suitable subsequence.
\begin{prop}
\label{prop:liminfEstimate} Suppose there exist sequences $\left(n_{k}\right)_{k\in\N}\in\N^{\N}$
and $\left(x_{k}\right)\in\R_{>0}^{\N}$ such that for all $k\in\N$,
$\J\left(\DD_{n_{k}}\right)<1/x_{k}$. Then we have 
\[
\underline{h}\leq\liminf_{k\rightarrow\infty}\frac{\log\left(\card\DD_{n_{k}}\right)}{\log\left(x_{k}\right)}.
\]
\end{prop}

\begin{proof}
Using $\max_{Q\in\DD_{n_{k}}}\J\left(Q\right)<1/x_{k}$ gives $M\left(x_{k}\right)\leq\card\left(\DD_{n_{k}}\right)$
and the claim follows by observing
\[
\underline{h}\leq\liminf_{k\rightarrow\infty}\frac{\log\left(M\left(x_{k}\right)\right)}{\log\left(x_{k}\right)}\leq\liminf_{k\rightarrow\infty}\frac{\log\left(\card\left(\DD_{n_{k}}\right)\right)}{\log\left(x_{k}\right)}.
\]
\end{proof}

\subsection{The dual problem\label{subsec:PartitionImprovemetBirmanSolomyak}}

This section is devoted to $\gamma_{n}\coloneqq\min_{P\in\Pi_{n}}\J\left(P\right).$
Using \prettyref{prop:GeneralUpperBounds}, we are able to extend
the class of set functions considered in \cite[Theorem 2.1]{MR0217487}
(i.\,e\@. we allow set functions $\J$ which are only assumed to
be non-trivial, non-negative, monotone and $\dim_{\infty}\left(\J\right)>0$).
Before we compare our results with the classical work, we provide
a proof of \prettyref{thm:ImproveBirmanSolomyak}.

\begin{proof}
[Proof of \prettyref{thm:ImproveBirmanSolomyak}] By the definition
of $\overline{h}$ in \prettyref{eq:Def_h_} we have for $h>\overline{h}$
and $n$ sufficiently large
\[
M\left(n^{1/h}\right)\leq n.
\]
This means that there exists $P\in\Pi_{n}$ such that $\J\left(P\right)<n^{-1/h}$
with $\card\left(P\right)\leq n$, and therefore, $\min_{P\in\Pi_{n}}\J\left(P\right)<n^{-1/h}$.
Thus, in tandem with \prettyref{thm:MainResult}, we see that $\overline{\alpha}\leq-1/\overline{h}=-1/\q$.
The upper bound $\overline{\alpha}\geq-1/\overline{h}$ holds clearly
for $\overline{\alpha}=0$. For $\overline{\alpha}\in\left[-\infty,0\right)$,
we choose $\alpha\in\left(\overline{\alpha},0\right)$. Then we have
\[
\min_{P\in\Pi_{n}}\J(P)<n^{\alpha}
\]
 for all large $n$. This implies $M\left(n^{-\alpha}\right)\leq n$,
which shows $\overline{h}\leq-1/\alpha$ and in particular for $\overline{\alpha}=-\infty$,
$\overline{h}=0$ and the upper bound follows. In the same way, one
shows $-1/\underline{h}=\underline{\alpha}$.
\end{proof}
For the remaining part of this section, we concentrate on special
choice $\DD=\D$ and $\J_{J,a}\left(Q\right)\coloneqq J\left(Q\right)\Lambda\left(Q\right)^{a}$,
$a>0$, $Q\in\D$, where $J$ is a non-trivial, non-negative, locally
non-vanishing, superadditive function on $\mathcal{D}$, that is,
if $Q\in\mathcal{D}$ is decomposed into a finite number of disjoint
cubes $\left(Q_{j}\right)_{j}$ of $\mathcal{D}$, then $\sum J\left(Q_{j}\right)\leq J\left(Q\right)$.
We are now interested in the decay rate of $\gamma_{\J_{J,a},n}$.
Upper estimates for $\gamma_{\J_{J,a},n}$ have first been obtained
in \parencite{MR0217487,Borzov1971}.

In the following we use the terminology as in \cite{MR4077830}. Let
$\Xi_{0}$ be a finite partition of $\Q$ of dyadic cubes from $\mathcal{D}$.
We say a partition $\Xi'$ of $\Q$ is an\emph{ elementary extension}
of $\Xi_{0}$ if it can be obtained by uniformly splitting some of
its cubes into $2^{d}$ equal sized disjoint cubes lying in $\mathcal{D}$.
We call a partition \emph{$\Xi$ dyadic subdivision }of an initial
partition \emph{$\Xi_{0}$ }if it is obtained from the partition \emph{$\Xi_{0}$}
with the help of a finite number of elementary extensions.
\begin{prop}
\label{prop:PartitionAlogrthmDueToSolandBirma}Let $\Xi_{0}$ be a
finite partition of $\Q$ with dyadic cubes from $\mathcal{D}$ and
suppose there exists $\varepsilon>0$ and a subset $\Xi_{0}^{'}\subset\Xi_{0}$
such that
\[
\sum_{Q\in\Xi_{0}\setminus\Xi_{0}^{'}}\Lambda\left(Q\right)\leq\varepsilon\:\text{\:and}\:\:\sum_{Q\in\Xi_{0}^{'}}J\left(Q\right)\leq\varepsilon.
\]
Let $\left(P_{k}\right)_{k\in\N}$ denote a sequence of dyadic partitions
obtained recursively as follows: set $P_{0}\coloneqq\Xi_{0}$ and,
for $k\in\N$, construct an elementary extension $P_{k}$ of $P_{k-1}$
by subdividing all cubes $Q\in P_{k-1}$, for which 
\[
\J_{J,a}\left(Q\right)\geq2^{-da}\eta_{a}\left(P_{k-1}\right)
\]
with $\eta_{a}\left(P_{k-1}\right)\coloneqq\J_{J,a}\left(P_{k-1}\right)$,
into $2^{d}$ equal sized cubes. Then, for all $k\in\N,$ we have
\[
\eta_{a}\left(P_{k}\right)=\J_{J,a}\left(P_{k}\right)\leq C\varepsilon^{\min(1,a)}\left(N_{k}-N_{0}\right)^{-(1+a)}J(\Q)
\]
with $N_{k}\coloneqq\card\left(P_{k}\right)$, $k\in\N_{0}$, and
the constant $C>0$ depends only on $a$ and $d$. In particular,
there exists $C'>0$ such that for all $n>N_{0}$,
\[
\gamma_{\J_{J,a},n}\leq C'J(\Q)\varepsilon^{\min(1,a)}n^{-(1+a)}.
\]
\end{prop}

\begin{proof}
A proof can be found in \cite{Borzov1971} or alternatively with further
details in \cite{elib_6573} based on the presentation of \cite{MR4077830}.
\end{proof}
\begin{defn}
\label{def:Singularfunction}We call $J$ a \emph{singular function
}with respect to $\Lambda$ if for every $\varepsilon>0$ there exists
a partitions $\Xi_{0}\subset\mathcal{D}$ of $\Q$ and a subset $\Xi'_{0}\subset\Xi_{0}$
such that
\[
\sum_{Q\in\Xi_{0}\setminus\Xi_{0}^{'}}\Lambda\left(Q\right)\leq\varepsilon\:\text{\:and}\:\:\sum_{Q\in\Xi_{0}^{'}}J\left(Q\right)\leq\varepsilon.
\]
\end{defn}

\begin{rem}
Since $\mathcal{D}$ is a semiring of sets, it follows that a measure
$\nu$ which is singular with respect to the Lebesgue measure, is
also singular as a function $J=\nu$ in the sense of \prettyref{def:Singularfunction}.
\end{rem}

As an immediate corollary of \prettyref{prop:PartitionAlogrthmDueToSolandBirma},
we obtain the following statement due to \cite{Borzov1971}.
\begin{cor}
We always have 
\[
\gamma_{\J_{J,a},n}=O\left(n^{-(1+a)}\right)\:\text{and}\:M_{\J_{J,a}}(x)=O\left(x^{1/(1+a)}\right).
\]
Additionally, if $J$ is singular, then
\[
\gamma_{\J_{J,a},n}=o\left(n^{-(1+a)}\right)\:\text{and}\:M_{\J_{J,a}}(x)=o\left(x^{1/(1+a)}\right).
\]
\end{cor}

\begin{rem}
\label{rem:BirmanImprovment}If $\GL_{\J_{J,a}}^{N}\left(q\right)<d(1-q(1+a))$
for some $q\in(0,1)$, then this estimate improves the corresponding
results of \cite[Theorem 2.1]{Borzov1971,MR0217487}, where only $\overline{\alpha}_{\J_{J,a}}\leq-(1+a)$
has been shown. Observe that $\GL_{\J_{J,a}}\left(q\right)=\GL_{J}\left(q\right)-adq$
for $q\geq0$ and $\GL_{J}(0)\leq d$. From the fact that $J$ is
superadditive, it follows that $\GL_{J}(1)\leq0$ and $q\mapsto\GL_{J}\left(q\right)$,
$q\geq0$ is decreasing. We only have to consider the case $\GL_{J}(1)>-\infty.$
Since $\GL_{J}$ is convex, for every $q\in[0,1]$, we deduce
\[
\GL_{\J_{J,a}}\left(q\right)=\GL_{J}\left(q\right)-adq\leq\GL_{J}(0)(1-q)-adq\leq d(1-q)-adq.
\]
This implies $\mathsf{\q}_{\J_{J,a}}\leq\GL_{J}(0)/(\GL_{J}(0)+ad)\leq1/(1+a)$.
From \prettyref{prop:GeneralUpperBounds} we deduce the improved upper
bounds
\[
\frac{-1}{\overline{h}_{\J_{J,a}}}=\frac{-1}{\q_{\J_{J,a}}}=\overline{\alpha}_{\J_{J,a}}\leq-\left(1+a\frac{d}{\overline{\dim}_{M}(J)}\right)\leq-(1+a).
\]
\end{rem}

\subsection{Coarse multifractal analysis \label{sec:Coarse-multifractal-analysis}}

Throughout this section let $\J$ be a non-trivial, non-negative,
monotone and locally non-vanishing set function defined on the set
of dyadic cubes $\mathcal{D}$ with $\dim_{\infty}\left(\J\right)>0$.

Recall the definition \ref{eq:Def:N_=00005Calpha(n)} of $N_{\alpha}$
and \ref{eq:Def_F_} of $\overline{F}$, $\underline{F}$.
\begin{lem}
For $\alpha\in\left(0,\dim_{\infty}\left(\J\right)\right)$ we have
\[
\overline{F}=\sup_{\alpha\geq\dim_{\infty}\left(\J\right)}\limsup_{n\to\infty}\frac{\log\left(N_{\alpha}\left(n\right)\right)}{\log\left(2^{n}\right)\alpha},\:\underline{F}=\sup_{\alpha\geq\dim_{\infty}\left(\J\right)}\liminf_{n\to\infty}\frac{\log\left(N_{\alpha}\left(n\right)\right)}{\log\left(2^{n}\right)\alpha}.
\]
\end{lem}

\begin{proof}
For fixed $\alpha\in\left(0,\dim_{\infty}\left(\J\right)\right)$,
by the definition of $\dim_{\infty}\left(\J\right)$ in \prettyref{eq:Def_Dim_infty},
for $n$ large we have $\J\left(\DD_{n}\right)\leq2^{-\alpha n}$.
For every $0<\alpha'<\alpha$, it follows that $N_{\alpha',\J}\left(n\right)=0.$
This proves that the supremum in the definition \prettyref{eq:Def_F_}
of $\overline{F}$ and $\underline{F}$ is obtained for $\alpha\geq\dim_{\infty}\left(\J\right)$
and the claim follows.
\end{proof}
We need the following elementary observation from large deviation
theory which seems not to be standard in the relevant literature.
\begin{lem}
\label{lem:exponential_decay-1} Suppose $\left(X_{n}\right)_{n\in\N}$
are real-valued random variables on some probability spaces $\left(\Omega_{n},\mathcal{A}_{n},\mu_{n}\right)$
such that the rate function $\mathfrak{c}\left(t\right)\coloneqq\limsup_{n\to\infty}\mathfrak{c}_{n}\left(t\right)$
is a proper convex function with $\mathfrak{c}_{n}\left(t\right)\coloneqq a_{n}^{-1}\log\int\exp tX_{n}\d\mu_{n}$,
$t\in\R$, $a_{n}\rightarrow\text{\ensuremath{\infty}}$ and such
that $0$ belongs to the interior of the domain of finiteness $\left\{ t\in\R\colon\mathfrak{c}\left(t\right)<\infty\right\} $.
Let $I=(a,d)$ be an open interval containing the subdifferential
$\partial\mathfrak{c}\left(0\right)=[b,c]$ of $\mathfrak{c}$ in
$0$. Then there exists $r>0$ such that for all $n$ sufficiently
large, 
\[
\mu_{n}\left(a_{n}^{-1}X_{n}\notin I\right)\leq2\exp\left(-ra_{n}\right).
\]
\end{lem}

\begin{proof}
We assume that $\partial\mathfrak{c}\left(0\right)=\left[b,c\right]$
and $I=\left(a,d\right)$ with $a<b\leq c<d$. First note that the
assumptions ensure that $-\infty<b\leq c<\infty.$ We have by the
Chebychev inequality for all $q>0$,
\begin{align*}
\mu_{n}\left(a_{n}^{-1}X_{n}\geq d\right) & =\mu_{n}\left(qX_{n}\geq qa_{n}d\right)\leq\exp\left(-qa_{n}d\right)\int\exp\left(qX_{n}\right)\d\mu_{n},
\end{align*}
 implying 
\[
\limsup_{n\to\infty}a_{n}^{-1}\log\mu_{n}\left(a_{n}^{-1}X_{n}\geq d\right)\leq\inf_{q>0}\mathfrak{c}\left(q\right)-qd=\inf_{q\in\R}\mathfrak{c}\left(q\right)-qd\leq0,
\]
where the equality follows from the assumption $c<d$, $\mathfrak{c}\left(0\right)=0$
and $\mathfrak{c}\left(q\right)-qd\geq(c-d)q\geq0$ for all $q\leq0$,
$\mathfrak{c}(0)=0$, and the continuity of $\mathfrak{c}$ at $0$.
Similarly, we find 
\[
\limsup_{n\to\infty}a_{n}^{-1}\log\mu_{n}\left(a_{n}^{-1}X_{n}\leq a\right)\leq\inf_{q<0}\mathfrak{c}\left(q\right)-qa=\inf_{q\in\R}\mathfrak{c}\left(q\right)-qa.
\]
We are left to show that both upper bounds are negative. We show the
first case by contradiction – the other case follows in exactly the
same way. Assuming $\inf_{q\in\R}\mathfrak{c}\left(q\right)-qd=0$
implies for all $q\in\R$ that $\mathfrak{c}\left(q\right)-qd\geq0$,
or after rearranging, $\mathfrak{c}\left(q\right)-\mathfrak{c}\left(0\right)\geq dq$.
This means, according to the definition of the sub-differential, that
$d\in\partial\mathfrak{c}\left(0\right)$, contradicting our assumptions.
\end{proof}
\begin{prop}
\label{prop:GeneralBound.-1}For a subsequence $(n_{k})$ define the
convex function on $\R_{\geq0}$ by $B\coloneqq\limsup_{k\to\infty}\GL_{n_{k}}$
and for some $q\geq0$, we assume $B\left(q\right)=\lim_{k\to\infty}\GL_{n_{k}}\left(q\right)$
and set $\left[a',b'\right]\coloneqq-\partial B\left(q\right)$. Then
we have $a'\geq\dim_{\infty}\left(\J\right)$ and
\begin{align*}
\frac{a'q+B\left(q\right)}{b'} & \leq\sup_{\alpha>b'}\liminf_{k\to\infty}\frac{\log\left(N_{\alpha}\left(n_{k}\right)\right)}{\alpha\log\left(2^{n_{k}}\right)}\\
 & \leq\sup_{\alpha\geq\dim_{\infty}\left(\J\right)}\liminf_{k\to\infty}\frac{\log\left(N_{\alpha}\left(n_{k}\right)\right)}{\alpha\log\left(2^{n_{k}}\right)}=\sup_{\alpha>0}\liminf_{k\to\infty}\frac{\log\left(N_{\alpha}\left(n_{k}\right)\right)}{\alpha\log\left(2^{n_{k}}\right)}.
\end{align*}
Moreover, if $B\left(q\right)=\GL\left(q\right),$ then $[a,b]=-\partial\GL\left(q\right)\supset-\partial B\left(q\right)$
and if additionally $0\leq q\leq\q$, then
\[
\frac{aq+\GL\left(q\right)}{b}\leq\frac{a'q+B\left(q\right)}{b'}.
\]
\end{prop}

\begin{proof}
Without loss of generality we can assume $b'<\infty.$ Moreover, $\dim_{\infty}\left(\J\right)>0$
implies $b'\geq a'\geq\dim_{\infty}\left(\J\right)>0$. Indeed, observe
that $B$ is again a convex function on $\R$. Thus, by the definition
of the sub-differential, we have for all $x>0$,
\[
B\left(q\right)-a'(x-q)\leq B(x)\leq\GL(x)\leq-x\dim_{\infty}\left(\J\right)+d,
\]
which gives $a'\ge\dim_{\infty}\left(\J\right)>0$. Let $q\geq0$.
Now, for all $k\in\N$ and $s<a'\leq b'<t$, we have with $L_{n_{k}}^{s,t}\coloneqq\left\{ Q\in\DD_{n_{k}}:2^{-sn_{k}}>\J\left(Q\right)>2^{-tn_{k}}\right\} $
\begin{align*}
N_{t,\J}^{\DD}\left(n_{k}\right) & \geq\card L_{n_{k}}^{s,t}\geq\sum_{Q\in L_{n_{k}}^{s,t}}\J\left(Q\right)^{q}2^{sn_{k}q}=2^{sn_{k}q+n_{k}\GL_{n_{k}}\left(q\right)}\sum_{Q\in\DD_{n_{k}}}\1_{L_{n_{k}}^{s,t}}(Q)\J\left(Q\right)^{q}2^{-n_{k}\GL_{n_{k}}\left(q\right)}\\
 & =2^{sn_{k}q+n_{k}\GL_{n_{k}}\left(q\right)}\left(1-\sum_{Q\in\DD_{n_{k}}}\1_{\left(L_{n_{k}}^{s,t}\right)^{\complement}}(Q)\J\left(Q\right)^{q}2^{-n_{k}\GL_{n_{k}}\left(q\right)}\right).
\end{align*}
We use the lower large deviation principle for the process $X_{k}\left(Q\right)\coloneqq\log\left(\J\left(Q\right)\right)$
with probability measure on $\DD_{n_{k}}$ given by $\mu_{k}\left(\left\{ Q\right\} \right)\coloneqq\J\left(Q\right)^{q}2^{-n_{k}\GL_{n_{k}}\left(q\right)}$.
We find for the free energy function
\begin{align*}
\mathfrak{c}\left(x\right) & \coloneqq\limsup_{k\to\infty}\frac{\log\left(\mathbb{E}_{\mu_{k}}\left(\exp\left(xX_{k}\right)\right)\right)}{\log\left(2^{n_{k}}\right)}=\limsup_{k\to\infty}\frac{1}{\log\left(2^{n_{k}}\right)}\log\left(\sum_{Q\in\DD_{n_{k}}}\J\left(Q\right)^{x+q}/2^{n_{k}\GL_{n_{k}}\left(q\right)}\right)\\
 & =\limsup_{k\to\infty}\GL_{n_{k}}(q+x)-B\left(q\right)=B(x+q)-B\left(q\right),
\end{align*}
with $-\partial\mathfrak{c}\left(0\right)=\left[a',b'\right]\subset(s,t)$
and hence there exists a constant $r>0$ depending on $s,t$ and $q$
such that for $k$ large by \prettyref{lem:exponential_decay-1}
\[
\sum_{Q\in\DD_{n_{k}}}\1_{\left(L_{n_{k}}^{s,t}\right)^{\complement}}(Q)\J\left(Q\right)^{q}/2^{n_{k}\GL_{n_{k}}\left(q\right)}=\mu_{k}\left(\frac{X_{k}}{\log\left(2^{n_{k}}\right)}\notin(-t,-s)\right)\leq2\exp\left(-rn_{k}\right).
\]
Therefore, $\liminf_{k\to\infty}\log\left(N_{t}^{\DD}\left(n_{k}\right)\right)/\log\left(2^{n_{k}}\right)\geq sq+B\left(q\right)$
for all $s<a'$ and $t>b'$ and hence
\[
\sup_{t>b'}\liminf_{k\to\infty}\frac{\log\left(N_{t}^{\DD}\left(n_{k}\right)\right)}{t\log\left(2^{n_{k}}\right)}\geq\sup_{t>b'}\frac{a'q+B\left(q\right)}{t}=\frac{a'q+B\left(q\right)}{b'}.
\]
The fact that $-\partial\GL\left(q\right)\supset-\partial B\left(q\right)$
if $\GL\left(q\right)=B\left(q\right)$ follows immediately from the
inequality $\limsup_{k\to\infty}\GL_{n_{k}}\leq\GL$.
\end{proof}
\begin{prop}
\label{prop:=00005CGL_reg.implies_lower_bound}If $\J$ is \emph{PF-regular
with respect to $\DD_{n}$, then} $\underline{F}=\q.$
\end{prop}

\begin{proof}
Due to \prettyref{prop:GeneralUpperBounds}, we can restrict our attention
to the case $\q>0$. First, assume $\GL\left(q\right)=\liminf_{n\to\infty}\GL_{n}\left(q\right)$
for $q\in\left(\q-\epsilon,\q\right)$, for some $\varepsilon>0$
and set $[a,b]=-\partial\GL\left(\q\right)$. Then by the convexity
of $\GL$ we find for every $\epsilon\in\left(0,\q\right)$ an element
$q\in\left(\q-\epsilon,\q\right)$ such that $\GL$ is differentiable
in $q$ with $-\left(\GL\right)'\left(q\right)\in[b,b+\varepsilon]$
since the points where $\GL$ is differentiable on $(0,\infty)$ lie
dense in $(0,\infty)$ which follows from the fact that $\GL$ is
a decreasing function and the fact that the left-hand derivative of
the convex function $\GL$ is left-hand continuous and non-decreasing.
Then we have by \prettyref{prop:GeneralBound.-1}
\begin{align*}
\sup_{\alpha\geq\dim(\J)}\liminf_{n\to\infty}\frac{\log^{+}\left(N_{\alpha}\left(n\right)\right)}{\alpha\log\left(2^{n}\right)} & \geq\sup_{\alpha>-\GL'\left(q\right)}\liminf_{n\to\infty}\frac{\log\left(N_{\alpha}\left(n\right)\right)}{\alpha\log\left(2^{n}\right)}\\
 & \geq\frac{-\GL'\left(q\right)q+\GL\left(q\right)}{-\GL'\left(q\right)}\geq\frac{b\left(\q-\epsilon\right)}{b+\varepsilon}.
\end{align*}
Taking the limit $\epsilon\to0$ proves the claim in this situation.
The case that $\GL$ exists as a limit in $\q$ and is differentiable
in $\q$ is covered by \prettyref{prop:GeneralBound.-1}.
\end{proof}
\begin{prop}
\label{prop:LowerBoundUpperSpecDim}We have $\overline{F}=\q.$
\end{prop}

\begin{proof}
Due to \prettyref{prop:GeneralUpperBounds}, we can restrict our attention
to the case $\q>0$. First note that by \prettyref{lem:uniformEstimatefortau_n},
for $n$ large, the family of convex functions $\left(\GL_{n}\right)$
restricted to $\left[0,\q+1\right]$ only takes values in $\left[-\left(\q+1\right)L+b,d\right]$
and on any compact interval $\left[c,e\right]\subset\left(0,\q+1\right)$
we have for all $c\leq x\leq y\leq e$

\[
\frac{\GL_{n}\left(x\right)-\GL_{n}\left(0\right)}{x-0}\leq\frac{\GL_{n}\left(y\right)-\GL_{n}\left(x\right)}{y-x}\leq\frac{\GL_{n}\left(\q+1\right)-\GL_{n}\left(y\right)}{\q+1-y}.
\]
We obtain by \prettyref{lem:uniformEstimatefortau_n} and the fact
that $\GL_{n}(0)\leq d$\textbf{
\[
\frac{\left(\q+1\right)L+b-d}{c}\leq\frac{\GL_{n}\left(x\right)-\GL_{n}\left(0\right)}{x-0}
\]
}and
\[
\frac{\GL_{n}\left(\q+1\right)-\GL_{n}\left(y\right)}{\q+1-y}\leq\frac{d-\left(\q+1\right)L-b}{\q+1-e},
\]
which implies 
\[
\left|\GL_{n}\left(y\right)-\GL_{n}\left(x\right)\right|\leq\max\left\{ \frac{|b|-\left(\q+1\right)L+d}{c},\frac{d-\left(\q+1\right)L+|b|}{\q+1-e}\right\} \left|x-y\right|
\]
and hence $\left(\GL_{n}|_{\left[c,e\right]}\right)$ is uniformly
bounded and uniformly Lipschitz and thus by Arzelà–Ascoli relatively
compact. Using this fact, we find a subsequence $\left(n_{k}\right)$
such that $\lim_{k\to\infty}\GL_{n_{k}}\left(\q\right)=\limsup_{n\to\infty}\GL_{n}\left(\q\right)=0$
and $\GL_{n_{k}}$ converges uniformly to the proper convex function
$B$ on
\[
\left[\q-\delta,\q+\delta\right]\subset\left(0,\q+1\right),
\]
for $\delta$ sufficiently small. We put $[a,b]\coloneqq-\partial B\left(\q\right)$.
Since the points where $B$ is differentiable are dense and since
$B$ is convex, we find for every $\delta>\epsilon>0$ an element
$q\in\left(\q-\varepsilon,\q\right)$ such that $B$ is differentiable
in $q$ with $-B'\left(q\right)\in[b,b+\epsilon]$. Noting $B\leq\GL$,
we have $-B'\left(q\right)\geq\dim_{\infty}\left(\J\right)$. Hence,
from \prettyref{prop:GeneralBound.-1} we deduce 
\[
\sup_{\alpha\geq\dim_{\infty}\left(\J\right)}\limsup_{n\to\infty}\frac{\log\left(N_{\alpha}\left(n\right)\right)}{\alpha\log\left(2^{n}\right)}\geq\sup_{\alpha>-B'\left(q\right)}\limsup_{k\to\infty}\frac{\log\left(N_{\alpha}\left(n_{k}\right)\right)}{\alpha\log\left(2^{n_{k}}\right)}\geq\frac{-B'\left(q\right)q+B\left(q\right)}{-B'\left(q\right)}\geq\frac{b\left(\q-\varepsilon\right)}{b+\varepsilon}.
\]
Taking the limit $\epsilon\to0$ gives the assertion.
\end{proof}
\begin{proof}[Proof of \prettyref{prop:LowerBoundLiminfFJ}]
 Let $q_{n}$ and $\partial c(0)\eqqcolon[a,b]$ with $a\leq b<0$
be given as stated in the remark. Since
\[
\card\left\{ Q\in\DD_{n}:\J(Q)\geq2^{-nb}\right\} \leq2^{q_{n}nb}\sum_{Q\in\DD_{n}:\J(Q)\geq2^{-nb}}\J(Q)^{q_{n}}\leq2^{q_{n}nb}
\]
we infer
\[
\frac{\log\left(\card\left\{ Q\in\DD_{n}:\J(Q)\geq2^{-nb}\right\} \right)}{\log\left(2^{bn}\right)}\leq q_{n}
\]
proving $\underline{\q}\geq\underline{F}_{\J}$. Further, for $0<s<t$
\begin{align*}
\card\left\{ Q\in\DD_{n}:\J(Q)\geq2^{-nt}\right\}  & \geq\card\left\{ Q\in\DD_{n}:2^{^{-nq_{n}s}}\geq\J(Q)^{q_{n}}\geq2^{-nq_{n}t}\right\} \\
 & \geq2^{^{nq_{n}s}}\sum_{Q\in\DD_{n}:2^{^{-ns}}\geq\J(Q)\geq2^{-nt}}\J(Q)^{q_{n}}.
\end{align*}
Define $X_{n}\left(Q\right)\coloneqq\log\left(\J(Q)\right)$, $Q\in\DD_{n}$
and $\mu_{n}\left(\left\{ Q\right\} \right)\coloneqq\J(Q)^{q_{n}}$,
$a_{n}=n\log(2)$, then the convex rate function is given by 
\[
q\mapsto\int\e^{qX_{n}}\d\mu_{n}=\e^{\log(2^{n})\tau_{\J,n}(q+q_{n})}.
\]
With $a\leq b<0$ as above, by \prettyref{lem:exponential_decay-1}
we find $r>0$ such that 
\[
\mu_{n}\left(\left\{ X_{n}\notin(a-\delta,b+\delta)\right\} \right)\leq2^{-rn+1}.
\]
Therefore, we obtain 
\begin{align*}
\card\left\{ Q\in\DD_{n}:\J(Q)\geq2^{n(a-\delta)}\right\}  & \geq\card\left\{ Q\in\DD_{n}:2^{n(b+\text{\ensuremath{\delta)}}}\geq\J(Q)\geq2^{n(a-\delta)}\right\} ,\\
 & \geq2^{^{nq_{n}(b+\delta)}}\left(1-2^{-nr+1}\right),
\end{align*}
implying 
\[
\frac{\log(\card\left\{ Q\in\DD_{n}:\J(Q)\geq2^{n(a-\delta)}\right\} }{\log(2^{n})(a-\delta))}\geq\frac{b+\delta}{a-\delta}q_{n}+\frac{\log\left(1-2^{-nr+1}\right)}{\log(2^{n})(\delta+b)}.
\]
Consequently, we have 
\[
\underline{F}_{\J}\geq\frac{b}{a}\,\underline{\q}.
\]
\end{proof}
\begin{example}
\label{exa:LowerBound}We consider a probability measure $\nu$ on
$\Q$ such that for all $Q,Q'\in\D_{n}$ we have $\nu(Q)=\nu(Q')$,
$n\in\N$ and 
\[
0<\underline{\dim}_{M}(\nu)<\overline{\dim}_{M}(\nu).
\]
Such a measure $\nu$ is provided in \cite[Example 5.5]{KN2022} (Homogeneous
Cantor measure with non-converging $L^{q}$-spectrum with $p_{1}=1/2$).
Now, for fixed $a>0$ we set $\J\left(Q\right)\coloneqq\J_{\nu,a/d,1}\left(Q\right)=\nu\left(Q\right)\Lambda\left(Q\right)^{a/d}$
as in \prettyref{eq:J_nu,a,b}. Then, with $c$ given as in \prettyref{prop:LowerBoundLiminfFJ},
we have $\partial c(0)=\left[-a-\overline{\dim}_{M}(\nu),-a-\underline{\dim}_{M}\left(\nu\right)\right]$
and 
\[
\underline{F}_{\J}=\frac{a+\underline{\dim}_{M}\left(\nu\right)}{a+\overline{\dim}_{M}\left(\nu\right)}\,\underline{\q}<\underline{\q}.
\]
To see this, note that for all $q>0$, we have
\[
\tau_{\J,n}(q)=\frac{\log\left(\sum_{Q\in\D_{n}}\J(Q)^{q}\right)}{\log(2^{n})}=q\frac{\log\left(\max_{Q\in\D_{n}}\J(Q)\right)}{\log\left(2^{n}\right)}+\tau_{\J,n}(0).
\]
Using $\tau_{\J,n}\left(q_{n}\right)=0$ implies 
\[
q_{n}=\frac{\log\left(2^{n}\right)\tau_{\J,n}(0)}{-\log\left(\max_{Q\in\D_{n}}\J(Q)\right)}=\frac{\tau_{\J,n}(0)}{a+\tau_{\J,n}(0)}=1-\frac{a}{a+\tau_{\J,n}(0)}.
\]
Since $\sum_{Q\in\D_{n}}\nu(Q)=\card\left\{ \nu(Q)>0:Q\in\D_{n}\right\} \max_{Q\in\D_{n}}\nu(Q)=1$
we find\textbf{ }
\[
\frac{\log\max_{Q\in\D_{n}}\J(Q)}{-\log\left(2^{n}\right)}+a=\frac{\log\tau_{\nu,n}\left(0\right)}{\log\left(2^{n}\right)}.
\]
Taking the limes inferior and using the fact that $\tau_{\J,n}(0)=\tau_{\nu,n}(0)$
then gives 
\begin{align}
\underline{\q} & =1-\frac{a}{a+\liminf_{n\rightarrow\infty}\tau_{\J,n}(0)}=\frac{\underline{\dim}_{M}\left(\nu\right)}{a+\underline{\dim}_{M}\left(\nu\right)}.\label{eq:lowerMinkowski_infinity_dim}
\end{align}
With this, we obtain
\[
\tau_{\J,n}\left(q+q_{n}\right)=\left(q+q_{n}\right)\frac{\log\left(\max_{Q\in\D_{n}}\J(Q)\right)}{\log\left(2^{n}\right)}+\tau_{\J,n}(0)=q\frac{\log\left(\max_{Q\in\D_{n}}\J(Q)\right)}{\log\left(2^{n}\right)},
\]
showing that 
\[
c\left(q\right)=\limsup_{n\rightarrow\infty}\tau_{\J,n}(q+q_{n})=\begin{cases}
-q\left(\overline{\dim}_{M}\left(\nu\right)+a\right), & q<\text{0},\\
-q\left(\underline{\dim}_{M}\left(\nu\right)+a\right), & q\geq0
\end{cases}
\]
and consequently $\partial c(0)=\left[-\overline{\dim}_{M}(\nu)-a,-\underline{\dim}_{M}(\nu)-a\right]$.
By \prettyref{prop:LowerBoundLiminfFJ}, 
\[
\underline{F}\geq\frac{\underline{\dim}_{M}\left(\nu\right)+a}{\overline{\dim}_{M}\left(\nu\right)+a}\,\underline{\q}.
\]
For the reverse inequality, note that for $\alpha>\overline{\dim}_{M}(\nu)+a$,
$n$ large and for all $Q\in\D_{n}$, we have $\J(Q)\geq2^{-n\alpha}$
and therefore,
\[
\card\left\{ Q\in\D_{n}:\J(Q)\geq2^{-\alpha n}\right\} =\card\left\{ Q\in\D_{n}:\J(Q)>0\right\} .
\]
For $0<\alpha<\overline{\dim}_{M}(\nu)+a$ there exists a subsequence
$\left(n_{k}\right)_{k}$ such that for all $Q\in\D_{n_{k}}$, we
have $\J(Q)\leq2^{-n_{k}\alpha}$ implying
\[
\liminf_{n\to\infty}\card\left\{ Q\in\D_{n}:\J(Q)\geq2^{-\alpha n}\right\} =0.
\]
Using \prettyref{eq:lowerMinkowski_infinity_dim} for the last equality,
we finally obtain
\begin{align*}
\underline{F} & =\sup_{\alpha>0}\liminf_{n\rightarrow\infty}\frac{\log^{+}\left(\card\left\{ Q\in\D_{n}:\J(Q)\geq2^{-\alpha n}\right\} \right)}{\log\left(2^{\alpha n}\right)}\\
 & =\sup_{\alpha\geq\overline{\dim}_{M}(\nu)+a}\liminf_{n\rightarrow\infty}\frac{\log^{+}\left(\card\left\{ Q\in\D_{n}:\J(Q)\geq2^{-\alpha n}\right\} \right)}{\log\left(2^{\alpha n}\right)}\\
 & \leq\liminf_{n\rightarrow\infty}\frac{\log^{+}\left(\card\left\{ Q\in\D_{n_{k}}:\J(Q)>0\right\} \right)}{\log\left(2^{n}\right)\left(\overline{\dim}_{M}(\nu)+a\right)}=\frac{\underline{\dim}_{M}\left(\nu\right)}{\overline{\dim}_{M}\left(\nu\right)+a}=\frac{\underline{\dim}_{M}\left(\nu\right)+a}{\overline{\dim}_{M}\left(\nu\right)+a}\,\mathfrak{\underline{q}}.
\end{align*}
\end{example}

\section{Proof of main results}

Now we are in a position to state the remaining proofs of our main
results from \prettyref{subsec:Main-results}. 
\begin{proof}
[Proof of \prettyref{lem:GxOptimal}] The equality $G_{x}=\mathcal{G}_{m_{x}+1}$
follows from the definitions. Clearly, we have $\inf\left\{ \card\left(P\right):P\in\Pi,\J\left(P\right)<1/x\right\} \leq\card\left(G_{x}\right)$,
since $G_{x}$ is a partition of $\Q$ which is ensured by the monotonicity
of $\J$ and the assumption that $\J$ is uniformly vanishing. For
the inverse inequality let $P_{\text{opt}}\in\Pi$ be the minimising
partition, i.\,e\@. we have $\inf\left\{ \card\left(P\right):P\in\Pi,\J\left(P\right)<1/x\right\} =\card\left(P_{\text{opt}}\right)$.
To prove that $P_{\text{opt}}=G_{x}$ we assume that there exists
$Q\in P_{\text{opt}}$ such that $Q\subset Q'\in\D_{\left|\log_{2}\Lambda\left(Q\right)\right|/d-1}$
with $\J(Q')<1/x$. Then, $\tilde{P}\coloneqq\left\{ Q''\in P_{\text{opt}}:Q''\cap Q'=\emptyset\right\} \cup Q'$
is also partition of $\Q$ with 
\[
\card(\tilde{P})<\card(\tilde{P})+2^{d}-1\leq\card\left(P_{\text{opt}}\right)
\]
and $\J(\tilde{P})<1/x$, contracting the assumption of $P_{\text{opt}}$
being minimising. Hence, we have $P_{\text{opt}}=G_{x}$.
\end{proof}
\begin{proof}
[Proof of \prettyref{lem:DualClassics}] Clearly, $\tilde{\Pi}_{n}\supset\Pi_{n}$
and hence $\inf_{P\in\tilde{\Pi}_{n}}\J\left(P\right)\leq\inf_{P\in\Pi_{n}}\J\left(P\right).$
Now suppose $\inf_{P\in\tilde{\Pi}_{n}}\J\left(P\right)=x$. Then
for every $\epsilon>0$ we have $M\left(x+\epsilon\right)=\card\left(G_{x+\epsilon}\right)\leq n$.
This shows that $\inf_{P\in\Pi_{n}}\J\left(P\right)\leq x+\epsilon$.
Since $\epsilon>0$ was arbitrary we conclude $\inf_{P\in\tilde{\Pi}_{n}}\J\left(P\right)\geq\inf_{P\in\Pi_{n}}\J\left(P\right)$.
\end{proof}
\begin{proof}
[Proof of \prettyref{thm:MainResult}] The main theorem is now a
consequence of \prettyref{prop:GeneralUpperBounds} and \prettyref{prop:LowerBoundUpperSpecDim}.
\end{proof}
\begin{proof}
[Proof of \prettyref{prop:Geometric_Bounsd}] The bounds are an
immediate consequence of the convexity of $\GL$, the fact that $-\dim_{\infty}\left(\J\right)$
is maximal asymptotic direction of $\GL$ and that $\GL\left(0\right)\leq\overline{\dim}_{M}\left(\J\right)$,
as shown in \prettyref{lem:GL(0)=00003DDim_M}. The case $\q>1$ is
portrayed in \prettyref{fig:Moment-generating-function}.
[Proof of \prettyref{thm:LqRegularImpliesRegular}] The theorem
is now a consequence of \prettyref{thm:MainResult} and \prettyref{prop:=00005CGL_reg.implies_lower_bound}.
\end{proof}
\printbibliography

\end{document}